\documentclass{amsart}

\usepackage[all]{xy}
\usepackage{amssymb,latexsym}
\usepackage{mltex}
\usepackage{amsmath}
\usepackage[colorlinks=true,linkcolor=blue,citecolor=blue]{hyperref}
 \usepackage{lscape}
\usepackage{enumerate}
\usepackage{amsthm}
\usepackage{hyperref}
\usepackage{multicol}

\newtheorem{thm}{Theorem}
\newtheorem{cor}{Corollary}
\newtheorem{lem}{Lemma}[section]
\newtheorem{prop}[lem]{Proposition}
\newtheorem{rem}{Remark}

\newcommand{\C}{\mathbb{C}}

\newcommand{\HH}{\mathbb{H}}

\newcommand{\R}{\mathbb{R}}

\newcommand{\beqt}{\begin{equation}}  \newcommand{\eeqt}{\end{equation}}
\newcommand{\bal}{\begin{align}}      \newcommand{\eal}{\end{align}}
\newcommand{\ba}{\begin{array}}      \newcommand{\ea}{\end{array}}
\newcommand{\bc}{\begin{center}}     \newcommand{\ec}{\end{center}}
\newcommand{\be}{\begin{enumerate}}  \newcommand{\ee}{\end{enumerate}}
\newcommand{\beq}{\begin{eqnarray}}  \newcommand{\eeq}{\end{eqnarray}}
\newcommand{\beQ}{\begin{eqnarray*}} \newcommand{\eeQ}{\end{eqnarray*}}
\newcommand{\bi}{\begin{itemize}}    \newcommand{\ei}{\end{itemize}}
\newcommand{\bt}{\begin{tabular}}    \newcommand{\et}{\end{tabular}}

\begin{document}
\title{Spinorial representation of submanifolds in metric Lie groups}
\author{Pierre Bayard, Julien Roth and Berenice Zavala Jim\'enez}

\maketitle

\begin{abstract}
In this paper we give a spinorial representation of submanifolds of any dimension and codimension into Lie groups equipped with left invariant metrics. As applications, we get a spinorial proof of the Fundamental Theorem for submanifolds into Lie groups, we recover previously known representations of submanifolds in $\R^n$ and in the 3-dimensional Lie groups $S^3$ and $E(\kappa,\tau),$ and we get a new spinorial representation for surfaces in the 3-dimensional semi-direct products: this achieves the spinorial representations of surfaces in the 3-dimensional homogeneous spaces. We finally indicate how to recover a Weierstrass-type representation for CMC-surfaces in 3-dimensional metric Lie groups recently given by Meeks, Mira, Perez and Ros.
\end{abstract}

\date{}
\maketitle\pagenumbering{arabic}
\section{Introduction}

The purpose of this paper is to give a spinorial representation of an isometric immersion of a Riemannian manifold $M$ into a Lie group $G$ equipped with a left invariant metric. The result is roughly the following: if $M$ is a simply connected Riemannian manifold, $E$ is a real vector bundle on $M$ equipped with a fibre metric and a compatible connection, and $B:TM\times TM\rightarrow E$ is bilinear and symmetric, then an isometric immersion of $M$ into $G$ with normal bundle $E$ and second fundamental form $B$ is equivalent to a spinor field $\varphi$ solution of a Killing-type equation on $M;$ the spinor bundle is constructed from the Clifford algebra of the metric Lie algebra $\mathcal{G}$ of the group, and the immersion is explicitly obtained by the integration of a $\mathcal{G}$-valued 1-form on $M$ defined in terms of the spinor field $\varphi.$ A precise statement with the suitable necessary hypotheses is given in Section \ref{section main result} of the paper.

The explicit representation formula of the immersion in terms of the spinor field may be considered as a generalized Weierstrass representation formula for manifolds into metric Lie groups. 

We then give some applications of this result. We first obtain an easy proof of a theorem by Piccione and Tausk \cite{PT}: under suitable hypotheses, the necessary equations of Gauss, Codazzi and Ricci are also sufficient to obtain an immersion of a simply connected manifold into a metric Lie group. We then show how our general result permits to recover the known spinorial representation for submanifolds in $\R^n$ \cite{BLR2}, and also obtain a new spinorial representation for submanifolds in $\HH^n$ considered as a metric Lie group. We finally study more precisely the case of surfaces in a 3-dimensional metric Lie group:  we recover the known spinorial representations in $S^3$ \cite{Mo} and $E(\kappa,\tau)$ \cite{Ro}, and obtain a new spinorial representation of surfaces in a general semi-direct product; this especially includes the cases of surfaces into the groups $Sol_3$ and $\HH^2\times\R,$ which achieves the spinorial representations of surfaces into the 3-dimensional homogeneous spaces initiated in \cite{Fr,Mo,Ro}. We also deduce alternative proofs of the Fundamental Theorems for surfaces in $E(\kappa,\tau)$ by Daniel \cite{Da} and in $Sol_3$ by Lodovici \cite{Lod}. We finish the paper showing how the general spinorial representation formula permits to recover the recent Weierstrass-type representation formula by Meeks, Mira, Perez and Ros \cite[Theorem 3.12]{MP} concerning constant mean curvature surfaces in 3-dimensional metric Lie groups.

The main result of the paper thus gives a general framework for a variety of Weierstrass-type representation formulas existing in the literature, and is also a tool to get representation formulas in new contexts. 

We quote the following related papers: Friedrich obtained in \cite{Fr} a geometric spinorial representation of a surface in $\R^3$ showing that a surface in $\R^3$ may be represented by a constant spinor field of $\R^3$ restricted to the surface; this result was then extended to $S^3$ and $\HH^3$ by Morel \cite{Mo} and to other 3-dimensional homogeneous spaces by Roth \cite{Ro}. It was then extended by Bayard, Lawn and Roth to surfaces in dimension 4 \cite{BLR} and afterwards to manifolds in $\R^n$ \cite{BLR2}. Spinorial representation were also studied in pseudo-Riemannian spaces, by Lawn in $\R^{2,1}$ \cite{La}, Lawn and Roth in 3-dimensional Lorentzian space forms \cite{LR}, Bayard in $\R^{3,1}$ \cite{Ba}, Bayard and Patty \cite{BP} and Patty \cite{Pa} in $\R^{2,2}.$ Close to the purpose of the paper, Berdinskii and Taimanov gave in \cite{BT} a spinorial representation for a surface in a 3-dimensional metric Lie group. 

The outline of the paper is as follows: Section \ref{sec preliminaries} is dedicated to preliminaries concerning notation and spin geometry of a submanifold in a metric Lie group, Section \ref{section main result} to the statement and the proof of the main theorem, and Section \ref{section fundamental theorem} to a spinorial proof of the Fundamental Theorem for submanifolds in a metric Lie group. We then give further applications in Section \ref{section special cases}: we study the cases of a submanifold in $\R^n$ and $\HH^n,$ and of a hypersurface in a general metric Lie group, specifying further to the cases of a surface in $S^3,$ $E(\kappa,\tau)$ and a semi-direct product, as $Sol_3$ and $\HH^2\times\R.$ We finally consider the case of a CMC-surface in a 3-dimensional metric Lie group. An appendix ends the paper concerning the links between the Clifford product and some natural operations on skew-symmetric operators.
\section{Preliminaries}\label{sec preliminaries}
\subsection{Notations}
Let $G$ be a Lie group, endowed with a left invariant metric $\langle.,.\rangle$, and $\mathcal{G}$ its Lie algebra: $\mathcal{G}$ is the space of the left invariant vector fields on $G$, equipped with the Lie bracket $[.,.]$ and is identified to the linear space tangent to $G$ at the identity. We consider the Maurer-Cartan form $\omega_G\in\Omega^1(G,\mathcal{G})$ defined by
\begin{equation}\label{def maurer-cartan}
\omega_G(v)={L_{g^{-1}}}_*(v)\hspace{.5cm}\in\ \mathcal{G}
\end{equation}
for all $v\in T_gG,$ where $L_{g^{-1}}$ denotes the left multiplication by $g^{-1}$ on $G$ and ${L_{g^{-1}}}_*:T_gG\rightarrow\mathcal{G}$ is its differential. This form induces a bundle isomorphism
\begin{eqnarray}\label{TG trivial}
TG&\rightarrow&G\times\mathcal{G}\\
(g,v)&\mapsto& (g,\omega_G(v)).\nonumber
\end{eqnarray}
which preserves the fibre metrics. We note that a vector field $X\in\Gamma(TG)$ is left invariant if, by (\ref{TG trivial}), $X:G\rightarrow\mathcal{G}$ is a constant map. Let us consider the Levi-Civita connection $\nabla^G$ of $(G,\langle.,.\rangle)$ and the linear map
\begin{eqnarray*}
\Gamma:\hspace{.5cm}\mathcal{G}&\rightarrow&\Lambda^2\mathcal{G}\\
X&\mapsto&\Gamma(X)
\end{eqnarray*}
such that, for all $X,Y\in\mathcal{G}$
\begin{equation}\label{def Gamma}
\nabla^G_XY=\Gamma(X)(Y).
\end{equation}
By the Koszul formula, $\Gamma$ is determined by the metric as follows: for all $X,Y,Z\in\mathcal{G},$ 
\begin{equation}\label{koszul formula}
\langle\Gamma(X)(Y),Z\rangle=\frac{1}{2}\langle[X,Y],Z\rangle+\frac{1}{2}\langle[Z,X],Y\rangle-\frac{1}{2}\langle[Y,Z],X\rangle.
\end{equation}
Since $\nabla^G$ is without torsion, we have, for all $X,Y\in\mathcal{G},$
\begin{equation}\label{nablaG without torsion}
\Gamma(X)(Y)-\Gamma(Y)(X)=[X,Y].
\end{equation}
We note that the curvature of $\nabla^G$ is given by
\begin{equation}\label{curvature nablaG}
R^G(X,Y)=\left[\Gamma(X),\Gamma(Y)\right]-\Gamma([X,Y])\hspace{.5cm}\in\Lambda^2\mathcal{G}
\end{equation}
for all $X,Y\in\mathcal{G}.$ In the formula the first brackets stand for the commutator of the endomorphisms.
\subsection{The spinor bundle of $G$}
Let us denote by $Cl(\mathcal{G})$ the Clifford algebra of $\mathcal{G}$ with its scalar product, and let us consider the representation
\begin{eqnarray*}
\rho:\hspace{.5cm} Spin(\mathcal{G})&\rightarrow& GL(Cl(\mathcal{G}))\\
a&\mapsto& \xi\mapsto a\xi.
\end{eqnarray*}
This representation is a real representation and is not irreducible in general: it is a sum of irreducible representations \cite{Lou}. By (\ref{TG trivial}) the principal bundle $Q_G$ of the positively oriented and orthonormal frames of $G$ is trivial
$$Q_G\simeq G\times SO(\mathcal{G}),$$
and we may consider the trivial spin structure
$$\tilde Q_G:= G\times Spin(\mathcal{G})$$
and the corresponding spinor bundle 
$$\Sigma:=\tilde Q_G\times_\rho Cl(G)\simeq G\times Cl(\mathcal{G}).$$
We will say that a spinor field $\varphi\in\Gamma(\Sigma)$ is \emph{left invariant} if it is constant as a map $G\rightarrow Cl(\mathcal{G}).$ The covariant derivative of a left invariant spinor field is
\begin{equation}\label{nabla spinor invariant}
\nabla^G_X\varphi=\frac{1}{2}\Gamma(X)\cdot\varphi
\end{equation}
where $\Gamma(X)\in\Lambda^2\mathcal{G}\subset Cl(\mathcal{G})$ and the dot "$\cdot$" stands for the Clifford product.

\subsection{The spin representation of $Spin(p)\times Spin(q)$}
Let us assume that $p+q=n,$ and fix an orthonormal basis $e^o_1,e^o_2,\ldots,e^o_n$ of $\mathcal{G};$ this gives a splitting $\mathcal{G}=\R^p\oplus\R^q$ (the first factor corresponds to the first $p$ vectors, and the second factor to the last $q$ vectors of the basis) and a natural map
$$Spin(p)\times Spin(q)\rightarrow Spin(\mathcal{G}),\hspace{1cm}(a_p,a_q)\mapsto a:=a_p\cdot a_q$$
associated to the isomorphism
$$Cl(\mathcal{G})=Cl_p\hat{\otimes} Cl_q.$$
We thus also have a representation, still denoted by $\rho,$
\begin{eqnarray}
\rho:\hspace{.5cm} Spin(p)\times Spin(q)&\rightarrow& GL(Cl(\mathcal{G}))\label{rep spin p spin q}\\
(a_p,a_q)&\mapsto& \xi\mapsto a\xi.\nonumber
\end{eqnarray}
\subsection{The twisted spinor bundle}\label{section twisted spinor bundle}
We consider a $p$-dimensional Riemannian manifold $M$ and a bundle $E\rightarrow M$ of rank $q,$ with a fibre metric and a compatible connection. We assume that $E$ and $TM$ are oriented and spin, with given spin structures
$$\tilde{Q}_M\stackrel{2:1}{\rightarrow} Q_M\hspace{.5cm}\mbox{and}\hspace{.5cm}\tilde{Q}_E\stackrel{2:1}{\rightarrow} Q_E$$
where $Q_M$ and $Q_E$ are the bundles of positively oriented orthonormal frames of $TM$ and $E,$ and we set
$$\tilde{Q}:=\tilde{Q}_M\times_M \tilde{Q}_E;$$
this is a $Spin(p)\times Spin(q)$ principal bundle. We define
$$\Sigma:=\tilde{Q}\times_{\rho} Cl(\mathcal{G})$$
and 
$$U\Sigma:=\tilde{Q}\times_{\rho} Spin(\mathcal{G})\hspace{.5cm}\subset\ \Sigma$$
where $\rho$ is the representation (\ref{rep spin p spin q}). Similarly to the usual construction in spin geometry, if we consider the representation
$$Ad:\hspace{.5cm}Spin(p)\times Spin(q)\rightarrow Spin(\mathcal{G})\stackrel{2:1}{\rightarrow} SO(\mathcal{G})\rightarrow GL(Cl(\mathcal{G}))$$
and the Clifford bundle
$$Cl(TM\oplus E)=\tilde{Q}\times_{Ad}Cl(\mathcal{G}),$$
there is a Clifford action of $Cl(TM\oplus E)$ on $\Sigma;$ this action will be denoted below by a dot $"\cdot".$ The vector bundle $\Sigma$ is moreover equipped with the covariant derivative $\nabla$ naturally associated to the spinorial connections on $\tilde{Q}_M$ and $\tilde{Q}_E.$ Let us denote by $\tau:Cl(\mathcal{G})\rightarrow Cl(\mathcal{G})$ the anti-automorphism of $Cl(\mathcal{G})$ such that
$$\tau(x_1\cdot x_2\cdots x_k)=x_k\cdots x_2\cdot x_1$$ 
for all $x_1,x_2,\ldots, x_k\in\mathcal{G},$ and set
\begin{eqnarray}
\langle\langle.,.\rangle\rangle:\hspace{.5cm}Cl(\mathcal{G})\times Cl(\mathcal{G})&\rightarrow& Cl(\mathcal{G})\label{def brackets 1}\\
(\xi,\xi')&\mapsto& \tau(\xi')\xi.\nonumber
\end{eqnarray}
This map is $Spin(\mathcal{G})-$invariant: for all $\xi,\xi'\in Cl(\mathcal{G})$ and $g\in Spin(\mathcal{G})$ we have
$$\langle\langle g\xi,g\xi'\rangle\rangle=\tau(g\xi')g\xi= \tau(\xi')\tau(g)g\xi= \tau(\xi')\xi=\langle\langle \xi,\xi'\rangle\rangle,$$
since $Spin(\mathcal{G})\subset \{g\in Cl^0(\mathcal{G}):\ \tau(g)g=1\};$ this map thus induces a $Cl(\mathcal{G})-$valued map
\begin{eqnarray}
\langle\langle.,.\rangle\rangle:\hspace{.5cm}\Sigma\times \Sigma&\rightarrow& Cl(\mathcal{G})\label{def brackets 2}\\
(\varphi,\varphi')&\mapsto& \langle\langle[\varphi],[\varphi']\rangle\rangle\nonumber
\end{eqnarray}
where $[\varphi]$ and $[\varphi']\in Cl(\mathcal{G})$ represent $\varphi$ and $\varphi'$ in some spinorial frame $\tilde{s}\in\tilde{Q}.$
\begin{lem}
The map $\langle\langle.,.\rangle\rangle:$ $\Sigma\times \Sigma\rightarrow Cl(\mathcal{G})$ satisfies the following properties: for all $\varphi,\psi\in\Gamma(\Sigma)$ and $X\in \Gamma(TM\oplus E),$
\begin{equation}\label{scalar product property1}
\langle\langle \varphi,\psi\rangle\rangle=\tau\langle\langle\psi,\varphi\rangle\rangle
\end{equation}
and
\begin{equation}\label{scalar product property2}
\langle\langle X\cdot\varphi,\psi\rangle\rangle=\langle\langle\varphi,X\cdot \psi\rangle\rangle.
\end{equation}
\end{lem}
\begin{proof}
We have
$$\langle\langle \varphi,\psi\rangle\rangle=\tau[\psi]\ [\varphi]=\tau(\tau[\varphi]\ [\psi])=\tau\langle\langle\psi,\varphi\rangle\rangle$$
and
$$\langle\langle X\cdot\varphi,\psi\rangle\rangle=\tau[\psi]\ [X][\varphi]=\tau([X][\psi])[\varphi]=\langle\langle\varphi,X\cdot \psi\rangle\rangle$$
where $[\varphi],$ $[\psi]$ and $[X]\ \in Cl(\mathcal{G})$ represent $\varphi,$ $\psi$ and $X$ in some given frame $\tilde{s}\in\tilde{Q}.$
\end{proof}
\begin{lem}
The connection $\nabla$ is compatible with the product $\langle\langle.,.\rangle\rangle:$
$$\partial_X\langle\langle\varphi,\varphi'\rangle\rangle=\langle\langle\nabla_X\varphi,\varphi'\rangle\rangle+\langle\langle\varphi,\nabla_X\varphi'\rangle\rangle$$
for all $\varphi,\varphi'\in\Gamma(\Sigma)$ and $X\in\Gamma(TM).$ 
\end{lem}
\begin{proof}
If $\varphi=[\tilde{s},[\varphi]]$ is a section of $\Sigma=\tilde{Q}\times_\rho Cl(\mathcal{G}),$ we have
\begin{equation}\label{nabla phi rho}
\nabla_X\varphi=\left[\tilde{s},\partial_X[\varphi]+\rho_*(\tilde{s}^*\alpha(X))([\varphi])\right],\hspace{1cm}\forall X\in\ TM,
\end{equation}
where $\rho$ is the representation (\ref{rep spin p spin q}) and $\alpha$ is the connection form on $\tilde{Q};$ the term $\rho_*(\tilde{s}^*\alpha(X))$ is an endomorphism of $Cl(\mathcal{G})$ given by the multiplication on the left by an element belonging to $\Lambda^2\mathcal{G}\subset Cl(\mathcal{G}),$ still denoted by  $\rho_*(\tilde{s}^*\alpha(X)).$ Such an element satisfies
$$\tau\left( \rho_*(\tilde{s}^*\alpha(X))\right)=-\rho_*(\tilde{s}^*\alpha(X)),$$
and we have
\begin{eqnarray*}
\langle\langle\nabla_X\varphi,\varphi'\rangle\rangle+\langle\langle\varphi,\nabla_X\varphi'\rangle\rangle&=&\tau\{[\varphi']\}\left(\partial_X[\varphi]+\rho_*(\tilde{s}^*\alpha(X))[\varphi]\right)\\
&&+\tau\left\{\partial_X[\varphi']+\rho_*(\tilde{s}^*\alpha(X))[\varphi']\right\}[\varphi]\\
&=&\tau\{[\varphi']\}\partial_X[\varphi]+\tau\left\{\partial_X[\varphi']\right\}[\varphi]\\
&=&\partial_X\langle\langle\varphi,\varphi'\rangle\rangle.
\end{eqnarray*}
\end{proof}
We finally note that there is a natural action of $Spin(\mathcal{G})$ on $U\Sigma,$ by right multiplication: for $\varphi=[\tilde s,[\varphi]]\in U\Sigma=\tilde{Q}\times_{\rho}Spin(\mathcal{G})$ and $a\in Spin(\mathcal{G})$ we set
\begin{equation}\label{def right action}
\varphi\cdot a:=[\tilde{s},[\varphi]\cdot a]\ \in U\Sigma.
\end{equation}

\subsection{The spin geometry of a submanifold of $G$} 
We keep the notation of the previous section, assuming moreover here that $M$ is a submanifold of a Lie group $G$ and that $E\rightarrow M$ is its normal bundle. If we consider spin structures on $TM$ and on $E$ whose sum is the trivial spin structure of $TM\oplus E$ \cite{Mi2}, we have
$$\Sigma=\tilde{Q}\times_{\rho} Cl(\mathcal{G})\simeq M\times Cl(\mathcal{G}),$$
where the last bundle is the spinor bundle of $G$ restricted to $M.$ Two connections are thus defined on $\Sigma,$ the connection $\nabla$ and the connection $\nabla^G;$ they satisfy the following Gauss formula:
\begin{equation}\label{gauss formula}
\nabla^G_X\varphi=\nabla_X\varphi+\frac{1}{2}\sum_{j=1}^p e_j\cdot B(X,e_j)\cdot\varphi
\end{equation} 
for all $\varphi\in\Gamma(\Sigma)$ and all $X\in \Gamma(TM),$ where $B:TM\times TM\rightarrow E$ is the second fundamental form of $M$ into $G$ and $e_1,\ldots,e_p$ is an orthonormal basis of $TM.$ We refer to \cite{Ba} for the proof (in a slightly different context). Since the covariant derivative of a left invariant spinor field is given by (\ref{nabla spinor invariant}), the restriction to $M$ of such a spinor field satisfies
\begin{equation}\label{gauss formula Gamma}
\nabla_X\varphi=-\frac{1}{2}\sum_{j=1}^p e_j\cdot B(X,e_j)\cdot\varphi+\frac{1}{2}\Gamma(X)\cdot\varphi
\end{equation} 
for all $X\in TM.$
\section{Main result}\label{section main result}
We consider a $p$-dimensional Riemannian manifold $M$ and a bundle $E\rightarrow M$ of rank $q,$ with a fibre metric and a compatible connection. We assume that $E$ and $TM$ are oriented and spin, with given spin structures, and consider the spinor bundles $\Sigma$ and $U\Sigma$ introduced in the previous section. We suppose that a bilinear and symmetric map $B:TM\times TM\rightarrow E$ is given, and we moreover do the following two assumptions:
\\
\begin{enumerate}
\item There exists a bundle isomorphism
\begin{equation}\label{def bundle iso f}
f:\ TM\oplus E\rightarrow M\times\mathcal{G}
\end{equation}
which preserves the metrics; this mapping permits to define a bundle map
\begin{equation}
\Gamma:\ TM\oplus E\rightarrow\Lambda^2(TM\oplus E)
\end{equation}
such that, for all $X,Y\in \Gamma(TM\oplus E),$
\begin{equation}\label{def Gamma TM+E}
f(\Gamma(X)(Y))=\Gamma(f(X))(f(Y))
\end{equation}
where on the right-hand side $\Gamma$ is the map defined on $\mathcal{G}$ by (\ref{def Gamma}), together with the following notion: a section $Z\in \Gamma(TM\oplus E)$ will be said to be left invariant if $f(Z):M\rightarrow\mathcal{G}$ is a constant map.
\item The covariant derivative of a left invariant section $Z\in \Gamma(TM\oplus E)$ is given by
\begin{equation}\label{nabla field invariant}
\nabla_XZ=\Gamma(X)(Z)-B(X,Z^T)+B^*(X,Z^N)
\end{equation}
for all $X\in TM,$ where  $Z=Z^T+Z^N$ in $TM\oplus E$ and $B^*:TM\times E\rightarrow TM$ is the bilinear map such that for all $X,Y\in\Gamma(TM)$ and $N\in\Gamma(E)$
$$\langle B(X,Y),N\rangle =\langle Y,B^*(X,N)\rangle.$$
\end{enumerate}
\begin{rem}
These two assumptions are equivalent to the assumptions made in \cite{Lod,PT}: they are necessary to write down the equations of Gauss, Codazzi and Ricci in a general metric Lie group, and to obtain a Fundamental Theorem for immersions in that context; see Section \ref{section fundamental theorem}.
\end{rem}
\begin{rem}\label{rmk cond frame}
Sometimes it is convenient to write these assumptions in some local frames. For sake of simplicity, we assume that $E$ is a trivial line bundle, oriented by a unit section $\nu.$ Let $(e_1^o,e_2^o,\ldots,e_n^o)$ be an orthonormal basis of $\mathcal{G}$ and $\Gamma_{ij}^k\in\R,$ $1\leq i,j,k\leq n,$ be such that
$$\Gamma(e_i^o)(e_j^o)=\sum_{k=1}^n\Gamma_{ij}^k\ e_k^o.$$
We set, for $i=1,\ldots,n,$ $\underline{e}_i\in\Gamma(TM\oplus E)$ such that $f(\underline{e}_i)=e_i^o,$ and $f_i\in C^{\infty}(M),$ $T_i\in\Gamma(TM)$ such that $\underline{e}_i=T_i+f_i\nu.$ Since $f$ preserves the metrics, the vectors $\underline{e}_1,\underline{e}_2,\ldots,\underline{e}_n$ are orthonormal, and we have
\begin{equation}\label{cond e_i orth}
\langle T_i,T_j\rangle+f_if_j=\delta_{ij}
\end{equation}
for all $i,j=1,\ldots,n.$ The assumption (\ref{nabla field invariant}) then reads as follows: for all $X\in TM$ and $j=1,\ldots,n,$
\begin{equation}\label{eqn 1 trad}
\nabla_XT_j=\sum_{i,k}\Gamma_{ij}^k\langle X,T_i\rangle T_k+f_jS(X),
\end{equation}
\begin{equation}\label{eqn 2 trad}
df_j(X)=\sum_{i,k}\Gamma_{ij}^kf_k\langle X,T_i\rangle-h(X,T_j)
\end{equation}
where $S(X)=B^*(X,\nu)$ and $h(X,Y)=\langle B(X,Y),\nu\rangle.$ Conversely, if vector fields $T_i\in\Gamma(TM)$ and functions $f_i\in C^{\infty}(M),$ $1\leq i\leq n,$ are given such that (\ref{cond e_i orth}), (\ref{eqn 1 trad}) and (\ref{eqn 2 trad}) hold, we may define a bundle isomorphism $f:TM\oplus E\rightarrow M\times\mathcal{G}$ preserving the metrics and such that (\ref{nabla field invariant}) holds: setting $\underline{e}_i=T_i+f_i\nu,$ we define $f$ such that $f(\underline{e}_i)=e_i^o,$ $i=1,\ldots,n.$
\end{rem}
We state the main result of the paper:
\begin{thm}\label{thm main result}
We moreover assume that $M$ is simply connected. The following statements are equivalent:
\begin{enumerate}
\item There exists a section $\varphi\in\Gamma(U\Sigma)$ such that
\begin{equation}\label{killing equation}
\nabla_X\varphi=-\frac{1}{2}\sum_{j=1}^pe_j\cdot B(X,e_j)\cdot\varphi+\frac{1}{2}\Gamma(X)\cdot\varphi
\end{equation}
for all $X\in TM.$
\item There exists an isometric immersion $F:\ M\rightarrow G$ with normal bundle $E$ and second fundamental form $B.$
\end{enumerate}
More precisely, if $\varphi$ is a solution of (\ref{killing equation}), replacing $\varphi$ by $\varphi\cdot a$ for some $a\in Spin(\mathcal{G})$ if necessary, and considering the $\mathcal{G}-$valued 1-form $\xi$ defined by
\begin{equation}\label{def xi}
\xi(X):=\langle\langle X\cdot\varphi,\varphi\rangle\rangle
\end{equation}
for all $X\in TM,$ the formula $F=\int\xi$  defines an isometric immersion in $G$ with normal bundle $E$ and second fundamental form $B.$ Here $\int$ stands for the Darboux integral, i.e. $F=\int\xi:M\rightarrow G$ is such that $F^*\omega_G=\xi,$ where $\omega_G\in\Omega^1(G,\mathcal{G})$ is the Maurer-Cartan form of $G$ defined in (\ref{def maurer-cartan}). Reciprocally, an isometric immersion $M\rightarrow G$ with normal bundle $E$ and second fundamental form $B$ may be written in that form.
\end{thm}
The formula $F=\int \xi$ where $\xi$ is defined by (\ref{def xi}) may be regarded as a generalized Weierstrass representation formula.
\\

This theorem generalizes the main result of \cite{BLR2} to a Lie group equipped with a left invariant metric (see Section \ref{section special cases}). 
\begin{rem}
If $\varphi$ is a solution of (\ref{killing equation}) and $a$ belongs to $Spin(\mathcal{G}),$ $\varphi':=\varphi\cdot a$ is also a solution of (\ref{killing equation}) (see (\ref{def right action}) for the definition of $\varphi\cdot a$). Moreover the associated 1-forms $\xi_{\varphi}$ and $\xi_{\varphi'}$ are linked by 
\begin{equation}\label{xi phi g}
\xi_{\varphi'}\ =\ \tau(a)\ \xi_{\varphi}\ a=Ad(a^{-1})\circ\xi_{\varphi}.
\end{equation}
Let us recall that a 1-form $\xi\in\Omega^1(M,\mathcal{G})$ is Darboux integrable if and only if it satisfies the structure equation $d\xi+[\xi,\xi]=0$ ($M$ is simply connected). The theorem thus says that if $\varphi$ is a solution of (\ref{killing equation}), it is possible to find an other solution $\varphi'$ of this equation such that $\xi_{\varphi'}$ is Darboux integrable and $F=\int\xi_{\varphi'}$ is an immersion with normal bundle $E$ and second fundamental form $B$. The proof of $(1)\Rightarrow (2)$ in the theorem will in fact follow these lines. See also Remark \ref{rmk congruence} below. 
\end{rem}
\begin{rem}
Setting
$$\vec{H}=\frac{1}{2}\sum_{j=1}^pB(e_j,e_j)\ \in E\hspace{.5cm}\mbox{and}\hspace{.5cm}\gamma=\frac{1}{2}\sum_{j=1}^pe_j\cdot\Gamma(e_j)\ \in Cl(TM\oplus E)$$
where $e_1,\ldots,e_p$ is an orthonormal basis of $TM,$ a solution $\varphi$ of (\ref{killing equation}) is a solution of the Dirac equation
\begin{equation}\label{dirac equation}
D\varphi:=\sum_{j=1}^pe_j\cdot\nabla_{e_j}\varphi=\left(\vec{H}+\gamma\right)\cdot\varphi.
\end{equation}
This equation will be especially interesting for the representation of a surface in a 3-dimensional Lie group (see Section \ref{section special cases}). 
\end{rem}
We now prove the theorem: $(1)\Rightarrow (2)$ will be a consequence of Propositions \ref{lem xi closed} and \ref{lem F isometry} below, and $(2)\Rightarrow (1)$ will be proved at the end of the section.
\begin{prop} \label{lem xi closed}
Assume that $\varphi\in\Gamma(U\Sigma)$ is a solution of (\ref{killing equation}) and define $\xi$ by (\ref{def xi}). Then
\begin{enumerate}
\item \label{lem xi closed 1} $\xi$ takes its values in $\mathcal{G}\subset Cl(\mathcal{G});$
\item \label{lem xi closed 2} there exists $T\in SO(\mathcal{G})$ such that $\xi=T\circ f;$ 
\item \label{lem xi closed 3} replacing $\varphi$ by $\varphi\cdot a$ where $a\in Spin(\mathcal{G})$ is such that $Ad(a)=T,$ we have $\xi=f,$ and $\xi$ satisfies the structure equation 
\begin{equation}\label{xi structure equation}
d\xi+[\xi,\xi]=0.
\end{equation}
\end{enumerate} 
\end{prop}
\begin{proof} 
(\ref{lem xi closed 1}). By the very definition of $\xi,$ we have
$$\xi(X)=\tau[\varphi][X][\varphi]$$  
for all $X\in TM,$ where $[X]$ and $[\varphi]$ represent $X$ and $\varphi$ in a given frame $\tilde{s}$ of $\tilde{Q}.$ Since $[X]$ belongs to $\mathcal{G}\subset Cl(\mathcal{G})$ and $[\varphi]$ is an element of $Spin(\mathcal{G}),$ $\xi(X)$ belongs to $\mathcal{G}.$
\\(\ref{lem xi closed 2}). Let us first show that for every left invariant section $Z\in\Gamma(TM\oplus E),$ the map $\xi(Z):M\rightarrow \mathcal{G}$ is constant: if $Z\in\Gamma(TM\oplus E)$ is left invariant, we compute, for $X\in TM,$
$$\partial_X\ \xi(Z)=\langle\langle\nabla_XZ\cdot\varphi,\varphi\rangle\rangle+\langle\langle Z\cdot\nabla_X\varphi,\varphi\rangle\rangle+\langle\langle Z\cdot\varphi,\nabla_X\varphi\rangle\rangle.$$
But, by (\ref{killing equation}), 
\begin{eqnarray}
\langle\langle Z\cdot\nabla_X\varphi,\varphi\rangle\rangle+\langle\langle Z\cdot\varphi,\nabla_X\varphi\rangle\rangle&=&\langle\langle[-\Gamma(X)+\sum_{j=1}^p e_j\cdot B(X,e_j),Z]\cdot\varphi,\varphi\rangle\rangle\nonumber\\
&=&\langle\langle \left\{-\Gamma(X)(Z)+B(X,Z^T)-B^*(X,Z^N)\right\}\cdot\varphi,\varphi\rangle\rangle\label{d xi computation}
\end{eqnarray}
where the brackets $[.,.]$ stand here for the commutator in $Cl(TM\oplus E)$ and where we use Lemmas \ref{lem1 ap1} and \ref{lem3 ap1} in the last step. Thus $\partial_X\ \xi(Z)=0$ by (\ref{nabla field invariant}), and $\xi(Z):M\rightarrow \mathcal{G}$ is constant. Now, if $(e_1^o,\ldots,e_n^o)$ is a fixed orthonormal basis of $\mathcal{G}$ and denoting by $\underline{e}_1,\ldots,\underline{e}_n$ the left invariant sections of $TM\oplus E$ such that $f(\underline{e}_i)=e_i^o,$ $i=1,\ldots,n,$ we have, for all section $Z=\sum_i Z_i\underline{e}_i\in\Gamma(TM\oplus E),$ 
$$\xi(Z)=\sum_{i=1}^n Z_i\ \xi(\underline{e}_i)$$ 
where $(\xi(\underline{e}_1),\ldots,\xi(\underline{e}_n))$ is a constant orthonormal basis of $\mathcal{G}.$ Considering the orthogonal transformation $T:\mathcal{G}\rightarrow \mathcal{G}$ such that $T(e_i^o)=\xi(\underline{e}_i),$ $i=1,\ldots,n,$ we get 
$$\xi(Z)=\sum_{i=1}^nZ_i\ T(e^o_i)=T\left(\sum_{i=1}^nZ_ie^o_i\right)=T(f(Z)),$$ 
i.e. $\xi=T\circ f.$ 
\\(\ref{lem xi closed 3}). For all $a\in Spin(\mathcal{G})$ and $X\in TM,$ we have
\begin{eqnarray*}
\langle\langle X\cdot(\varphi\cdot a),\varphi\cdot a\rangle\rangle&=&\tau([\varphi]a)[X][\varphi]a\\
&=&\tau(a)\ \langle\langle X\cdot\varphi,\varphi\rangle\rangle\ a\\
&=&Ad(a^{-1})(\xi(X))\\
&=&Ad(a^{-1})(T\circ f(X));
\end{eqnarray*}
thus, replacing $\varphi$ by $\varphi\cdot a$ where  $a\in Spin(\mathcal{G})$ is such that $Ad(a)=T$ we get $\xi=f.$ By the computation in (\ref{d xi computation}), we have, for $X,Y\in \Gamma(TM)$ such that $\nabla X=\nabla Y=0$ at $x_0,$
\begin{eqnarray*}
\partial_X\ \xi(Y)&=&\langle\langle Y\cdot\nabla_X\varphi,\varphi\rangle\rangle+\langle\langle Y\cdot\varphi,\nabla_X\varphi\rangle\rangle\\
&=&\langle\langle \left\{-\Gamma(X)(Y)+B(X,Y)\right\}\cdot\varphi,\varphi\rangle\rangle
\end{eqnarray*}
and thus
\begin{eqnarray*}
d\xi(X,Y)&=&\partial_X\ \xi(Y)-\partial_Y\ \xi(X)\\
&=&-\langle\langle \left\{\Gamma(X)(Y)-\Gamma(Y)(X)\right\}\cdot\varphi,\varphi\rangle\rangle\\
&=&-\xi(\Gamma(X)(Y)-\Gamma(Y)(X))\\
&=&-[\xi(X),\xi(Y)],
\end{eqnarray*}
since $\xi=f,$ $\Gamma$ satisfies (\ref{def Gamma TM+E}), and by (\ref{nablaG without torsion}).
\end{proof}
We keep the assumption and notation of Proposition \ref{lem xi closed}, and moreover assume that $M$ is simply connected; we consider 
$$F:M\rightarrow G$$ 
such that $F^*\omega_G=\xi$ (assuming that $\varphi$ is chosen in such a way that $\xi$ satisfies the structure equation (\ref{xi structure equation})). The next proposition follows from the properties of the Clifford product:
\begin{prop}\label{lem F isometry}
1. The map $F:M\rightarrow G$ is an isometry.
\\2. The map
\begin{eqnarray*}
\Phi_E:\hspace{1cm} E&\rightarrow& F(M)\times\mathcal{G}\\
X\in E_m&\mapsto& (F(m),\xi(X)) 
\end{eqnarray*}
is an isometry between $E$ and the normal  bundle of $F(M)$ into $G,$ preserving connections and second fundamental forms. Here, for $X\in E,$ $\xi(X)$ still stands for the quantity $\langle\langle X\cdot\varphi,\varphi\rangle\rangle.$
\end{prop}
\begin{proof}
For $X,Y\in\Gamma(TM\oplus E),$ we have
\begin{eqnarray*}
\langle \xi(X),\xi(Y)\rangle &=&-\frac{1}{2}\left(\xi(X)\xi(Y)+\xi(Y)\xi(X)\right)\\
&=&-\frac{1}{2}\left(\tau[\varphi][X][\varphi]\tau[\varphi][Y][\varphi]+\tau[\varphi][Y][\varphi]\tau[\varphi][X][\varphi]\right)\\
&=&-\frac{1}{2}\tau[\varphi]\left([X][Y]+[Y][X]\right)[\varphi]\\
&=&\langle X,Y\rangle,
\end{eqnarray*}
since $[X][Y]+[Y][X]=-2\langle[X],[Y]\rangle=-2\langle X,Y\rangle.$ This implies that $F$ is an isometry, and that $\Phi_E$ is a bundle map between $E$ and the normal bundle of $F(M)$ into $G$ which preserves the metrics of the fibers. Let us denote by $B_F$ and $\nabla'^F$ the second fundamental form and the normal connection of the immersion $F;$ the aim is now to prove that
\begin{equation}\label{xi preserves ff connection}
\xi(B(X,Y))=B_F(\xi(X),\xi(Y))\hspace{.5cm}\mbox{and}\hspace{.5cm}\xi(\nabla'_XN)=\nabla'^F_{\xi(X)}\xi(N)
\end{equation}
for $X,Y\in \Gamma(TM)$ and $N\in\Gamma(E).$ First,
$$B_F(\xi(X),\xi(Y))=(\nabla^G_{\xi(X)}\xi(Y))^N=\left\{\partial_X\ \xi(Y)+\Gamma(\xi(X))(\xi(Y))\right\}^N$$
where the superscript $N$ means that we consider the component of the vector which is normal to the immersion. We fix a point $x_0\in M,$ assume that $\nabla Y=0$ at $x_0,$ and compute, using (\ref{d xi computation}):
\begin{eqnarray*}
\partial_X\ \xi(Y)&=&\langle\langle Y\cdot\nabla_X\varphi,\varphi\rangle\rangle+\langle\langle Y\cdot\varphi,\nabla_X\varphi\rangle\rangle\\
&=&\langle\langle B(X,Y)\cdot\varphi,\varphi\rangle\rangle-\langle\langle\Gamma(X)(Y)\cdot\varphi,\varphi\rangle\rangle.
\end{eqnarray*}  
Since $\langle\langle B(X,Y)\cdot\varphi,\varphi\rangle\rangle=\xi(B(X,Y))$ is normal to the immersion, we get
$$\{\partial_X\ \xi(Y)\}^N=\xi(B(X,Y))-\langle\langle \Gamma(X)(Y)\cdot\varphi,\varphi\rangle\rangle^N,$$
and thus
\begin{eqnarray*}
B_F(\xi(X),\xi(Y))&=&\xi(B(X,Y))-\langle\langle \Gamma(X)(Y)\cdot\varphi,\varphi,\rangle\rangle^N+\Gamma(\xi(X))(\xi(Y))^N\\
&=&\xi(B(X,Y))
\end{eqnarray*}
since 
\begin{eqnarray*}
\langle\langle\Gamma(X)(Y)\cdot\varphi,\varphi\rangle\rangle&=&\xi(\Gamma(X)(Y))\\
&=&f(\Gamma(X)(Y))\\
&=&\Gamma(f(X))(f(Y))\hspace{.5cm}(\mbox{by definition of }\Gamma\mbox{ on }TM\oplus E)\\
&=&\Gamma(\xi(X))(\xi(Y)).
\end{eqnarray*}

We finally show the second identity in (\ref{xi preserves ff connection}): we have
\begin{eqnarray*}
\nabla'^F_{\xi(X)}\xi(N)&=&(\nabla^G_{\xi(X)}\xi(N))^N\\
&=&(\partial_X\ \xi(N)+\Gamma(\xi(X))(\xi(N)))^N\\
&=&\langle\langle\nabla'_XN\cdot\varphi,\varphi\rangle\rangle^N+\langle\langle N\cdot\nabla_X\varphi,\varphi\rangle\rangle^N+\langle\langle N\cdot\varphi,\nabla_X\varphi\rangle\rangle^N\\
&&+\Gamma(\xi(X))(\xi(N))^N.
\end{eqnarray*}
The first term in the right-hand side is $\xi(\nabla'_XN),$ and we only need to show that
\begin{equation}\label{lem F isometry expr}
\langle\langle N\cdot\nabla_X\varphi,\varphi\rangle\rangle^N+\langle\langle N\cdot\varphi,\nabla_X\varphi\rangle\rangle^N+\Gamma(\xi(X))(\xi(N)))^N=0.
\end{equation}
From (\ref{d xi computation}), we have
$$\langle\langle N\cdot\nabla_X\varphi,\varphi\rangle\rangle+\langle\langle N\cdot\varphi,\nabla_X\varphi\rangle\rangle=-\langle\langle B^*(X,N)\cdot\varphi,\varphi\rangle\rangle-\langle\langle\Gamma(X)(N)\cdot\varphi,\varphi\rangle\rangle,$$
which gives (\ref{lem F isometry expr}) since $\langle\langle B^*(X,N)\cdot\varphi,\varphi\rangle\rangle$ is tangent to the immersion ($B^*(X,N)$ belongs to $TM$) and
$$\langle\langle \Gamma(X)(N)\cdot\varphi,\varphi\rangle\rangle=\Gamma(\xi(X))(\xi(N))$$
(see the first part of the proof above).
\end{proof}

We finally show the converse statement $(2)\Rightarrow (1):$ we suppose that $F:M\rightarrow G$ is an isometric immersion with normal bundle $E$ and second fundamental form $B,$ we consider the orthonormal frame $s_o=\textit{1}_{SO(\mathcal{G})}$ of $\mathcal{G}$, and the spinor frame $\tilde{s}_o=\textit{1}_{Spin(\mathcal{G})}$ (recall that $Q_G=G\times SO(\mathcal{G})$ and $\tilde{Q}_G=G\times Spin(\mathcal{G})$; see Section \ref{sec preliminaries}). The spinor field $\varphi=[\tilde{s}_o,\textit{1}_{Cl(\mathcal{G})}]$ satisfies (\ref{killing equation}) as a consequence of the Gauss formulas (\ref{gauss formula})-(\ref{gauss formula Gamma}); moreover, its associated 1-form is, for all $X\in TM,$
$$\xi(X)=\langle\langle F_*X\cdot\varphi,\varphi\rangle\rangle=\tau [\varphi]\ [F_*X]\ [\varphi]=[F_*X],$$
where $[F_*X]\in\mathcal{G}$ represents $F_*X$ in $s_o,$ that is $[F_*X]=\omega_G(F_*X)$ ($\omega_G\in\Omega^1(G,\mathcal{G})$ is the Maurer-Cartan form of $G$). Thus $\xi=F^*\omega_G,$ that is $F=\int\xi.$ 
\begin{rem}\label{rmk congruence}
We proved in Proposition \ref{lem F isometry} that if $\varphi\in\Gamma(U\Sigma)$ is a solution of (\ref{killing equation}) such that $\xi_{\varphi}$ satisfies the structure equation (\ref{xi structure equation}) then $F=\int\xi_{\varphi}$ is an immersion with normal bundle $E$ and second fundamental form $B.$ By (\ref{xi phi g}) it is clear that if $a\in Spin(\mathcal{G})$ is such that $Ad(a^{-1}):\mathcal{G}\rightarrow\mathcal{G}\ \in SO(\mathcal{G}) $ is an automorphism of Lie algebra, then $\xi_{\varphi\cdot a}$ satisfies the structure equation too; in fact, the corresponding immersions $F_{\varphi}=\int\xi_\varphi$ and $F_{\varphi\cdot a}=\int\xi_{\varphi\cdot a}$ are linked by the following formula: if $\Phi_a:G\rightarrow G$ is the automorphism of $G$ such that $d(\Phi_a)_e=Ad(a^{-1}),$ then $\Phi_a$ is also an isometry for the left invariant metric, and
\begin{equation}\label{expr F phi g}
F_{\varphi\cdot a}=L_{b}\circ\Phi_a\circ F_{\varphi}
\end{equation}
for some $b$ belonging to $G.$ This relies on the following formula: if $\Phi:G\rightarrow G$ is an automorphism, $\omega_G\in\Omega^1(G,\mathcal{G})$ is the Maurer-Cartan form of $G$ and $F:M\rightarrow G$ is a smooth map, then 
$$(\Phi\circ F)^*\omega_G=d(\Phi)_e\circ (F^*\omega_G).$$
This formula applied to $\Phi=\Phi_a$ and $F=F_\varphi$ shows that $\Phi_a\circ F_{\varphi}$ is a solution of the Darboux equation associated to the form $\xi_{\varphi\cdot a};$ thus, by uniqueness of a solution of the Darboux equation, (\ref{expr F phi g}) holds for some $b$ belonging to $G.$
\end{rem}

\section{An application: the Fundamental Theorem for immersions in a metric Lie group}\label{section fundamental theorem}
We now show that the equations of Gauss, Ricci and Codazzi on $B$ are exactly the integrability conditions of (\ref{killing equation}). We recall these equations for immersions in the metric Lie group $G$: if $R^G$ denotes the curvature tensor of $(G,\langle.,.\rangle),$ and if $R^T$ and $R^N$ stand for the curvature tensors of the connections on $TM$ and on $E$ ($M$ is a submanifold of $G$ and $E$ is its normal bundle), then we have, for all $X,Y,Z\in\Gamma(TM)$ and $N\in\Gamma(E),$
\begin{enumerate}
\item the Gauss equation
\begin{equation}\label{Gauss equation G}
(R^G(X,Y)Z)^T=R^T(X,Y)Z-B^*(X,B(Y,Z))+B^*(Y,B(X,Z)),
\end{equation}
\item the Ricci equation
\begin{equation}\label{Ricci equation G}
(R^G(X,Y)N)^N=R^N(X,Y)N-B(X,B^*(Y,N))+B(Y,B^*(X,N)),
\end{equation}
\item the Codazzi equation
\begin{equation}\label{Codazzi equation G}
(R^G(X,Y)Z)^N=\tilde{\nabla}_X B(Y,Z)-\tilde{\nabla}_Y B(X,Z);
\end{equation}
\end{enumerate}
in the last equation, $\tilde{\nabla}$ denotes the natural connection on $T^*M\otimes T^*M\otimes E.$\\

These equations make sense if $M$ is an abstract manifold and $E\rightarrow M$ is an abstract bundle as in Section \ref{section main result}, if we assume the existence of the bundle map $f$ in (\ref{def bundle iso f}), since $f$ permits to define $\Gamma$ on $TM\oplus E$ by (\ref{def Gamma TM+E}), and $R^G$ may be written in terms of $\Gamma$ only (see (\ref{nablaG without torsion})-(\ref{curvature nablaG})). We prove the following:
\begin{prop}
We assume that $M$ is simply connected. There exists $\varphi\in\Gamma(U\Sigma)$ solution of (\ref{killing equation}) if and only if $B:TM\times TM\rightarrow E$ satisfies the Gauss, Ricci and Codazzi equations.
\end{prop}
\begin{proof}
We first prove that the Gauss, Ricci and Codazzi equations are necessary if we have a non-trivial solution of (\ref{killing equation}). We assume that $\varphi\in\Gamma(U\Sigma)$ is a solution of (\ref{killing equation}) and compute the curvature 
\begin{equation*}
R(X,Y)\varphi=\nabla_X\nabla_Y\varphi-\nabla_Y\nabla_X\varphi-\nabla_{[X,Y]}\varphi.
\end{equation*}
We fix a point $x_0\in M,$ and assume that $\nabla X=\nabla Y=0$ at $x_0.$ We have
\begin{eqnarray*}
\nabla_X\nabla_Y\varphi&=&-\frac{1}{2}\sum_{j=1}^pe_j\cdot\left(\tilde{\nabla}_X B(Y,e_j)\cdot\varphi+B(Y,e_j)\cdot\nabla_X\varphi\right)\\
&&+\frac{1}{2}\left(\nabla_X\Gamma(Y)\cdot\varphi+\Gamma(Y)\cdot\nabla_X\varphi\right)\\
&=&-\frac{1}{2}\sum_{j=1}^pe_j\cdot\tilde{\nabla}_X B(Y,e_j)\cdot\varphi-\frac{1}{4}\sum_{j,k=1}^pe_j\cdot e_k\cdot B(Y,e_j)\cdot B(X,e_k)\\
&&-\frac{1}{4}\sum_{j=1}^p e_j\cdot B(Y,e_j)\cdot\Gamma(X)\cdot\varphi+\frac{1}{2}\nabla_X\Gamma(Y)\cdot\varphi-\frac{1}{4}\Gamma(Y)\cdot\sum_{j=1}^pe_j\cdot B(X,e_j)\cdot\varphi\\
&&+\frac{1}{4}\Gamma(Y)\cdot\Gamma(X)\cdot\varphi.
\end{eqnarray*}
Thus
\begin{eqnarray}
R(X,Y)\varphi&=&-\frac{1}{2}\sum_{j=1}^pe_j\cdot\left(\tilde{\nabla}_X B(Y,e_j)-\tilde{\nabla}_Y B(X,e_j)\right)\cdot\varphi\nonumber\\
&&\underbrace{+\frac{1}{4}\sum_{j\neq k} e_j\cdot e_k\cdot \left(B(X,e_j)\cdot B(Y,e_k)-B(Y,e_j)\cdot B(X,e_k)\right)}_{\mathcal{A}}\cdot\varphi\label{R function B}\\
&&\underbrace{-\frac{1}{4}\sum_{j=1}^p\left(B(X,e_j)\cdot B(Y,e_j)-B(Y,e_j)\cdot B(X,e_j)\right)}_{\mathcal{B}}\cdot\varphi\nonumber\\
&&\underbrace{+\frac{1}{2}\left[\sum_{j=1}^pe_j\cdot B(X,e_j),\Gamma(Y)\right]}_{\mathcal{C}_1}\cdot \varphi +\underbrace{-\frac{1}{2}\left[\sum_{j=1}^pe_j\cdot B(Y,e_j),\Gamma(X)\right]}_{\mathcal{C}_2}\cdot \varphi \nonumber \\
&& \underbrace{+{\frac{1}{2}(\nabla_X\Gamma(Y)-\nabla_Y\Gamma(X))}}_{\mathcal{C}_3}\cdot \varphi+\underbrace{ -\frac{1}{2} [\Gamma(X),\Gamma(Y)]}_{\mathcal{C}_4}\cdot \varphi \nonumber
\end{eqnarray}
where the brackets stand for the commutator in the Clifford bundle $Cl(TM\oplus E):$ $\forall\eta,\xi\ \in Cl(TM\oplus E),$
$$[\eta,\xi]=\frac{1}{2}\left(\eta\cdot\xi-\xi\cdot\eta\right).$$
We computed the second and the third terms in \cite{BLR2}; we only recall the result here:
\begin{lem}\cite{BLR2} \label{computation AB} 
We have
$$\mathcal{A}=\frac{1}{2}\sum_{j< k}\left\{\langle B^*(X,B(Y,e_j)),e_k\rangle- \langle B^*(Y,B(X,e_j)),e_k\rangle\right\}e_j\cdot e_k$$
and
$$\mathcal{B}=\frac{1}{2}\sum_{k<l}\left\langle B(X,B^*(Y,n_k))-B(Y,B^*(X,n_k)),n_l\right\rangle n_k\cdot n_l,$$
where $e_1,\ldots,e_p$ and $n_1,\ldots,n_q$ are orthonormal bases of $TM$ and $E.$
\end{lem}
We now compute the other terms in (\ref{R function B}). We first compute the covariant derivative of $\Gamma,$ considering $\Gamma$ as a map
$$\Gamma:\hspace{.5cm} TM\oplus E\rightarrow End(TM\oplus E).$$
\begin{lem}\label{lemma formula nabla gamma}
If $X,Y\in TM$ and $Z\in TM\oplus E,$
\begin{eqnarray*}
(\nabla_X\Gamma)(Y)Z&=&\left\{\Gamma(X)\circ\Gamma(Y)-\Gamma(Y)\circ\Gamma(X)\right\}(Z)-\Gamma(\Gamma(X)Y)(Z)+\Gamma(B(X,Y))(Z)\\
&&-B(X,(\Gamma(Y)Z)^{T})+B^{*}(X,(\Gamma(Y)Z)^{N})+\Gamma(Y)(B(X,Z^T)-B^{*}(X,Z^N)).
\end{eqnarray*}
\end{lem}
\begin{proof}
Since the expression is tensorial, we may assume that $X,Y,Z\in \Gamma(TM\oplus E)$ are left invariant vector fields. By definition,
\begin{equation}\label{edo1_nabla}
(\nabla_{X}\Gamma)(Y)Z=\nabla_X(\Gamma(Y)Z) -\Gamma(\nabla_X Y)Z-\Gamma(Y)(\nabla_X Z).
\end{equation}
Since $X, Y$ and $Z$ are left invariant vector fields, so are $\Gamma(Y)Z,$ $\nabla_X Y$ and $\nabla_X Z,$ and, by (\ref{nabla field invariant}), 
 \begin{eqnarray*}
\nabla_X(\Gamma(Y)Z)=\Gamma(X)(\Gamma(Y)Z)-B(X,(\Gamma(Y)Z)^{T})+B^*(X,(\Gamma(Y)Z)^{N}),
\end{eqnarray*}
\begin{eqnarray*}
\Gamma(Y)(\nabla _X Z)=\Gamma(Y)(\Gamma(X)Z)-\Gamma(Y)B(X,Z^T)+\Gamma(Y)B^*(X,Z^N)
\end{eqnarray*}
and
\begin{eqnarray*}
\Gamma(\nabla _X Y)(Z)=\Gamma(\Gamma(X)Y)Z-\Gamma(B(X,Y^T))Z+\Gamma(B^*(X,Y^N))Z.
\end{eqnarray*}
Plugging these formulas in (\ref{edo1_nabla}) and using finally that $Y$ belongs to $TM$ (ie $Y^T=Y$ and $Y^N=0$), we get the result.
\end{proof}
We now regard $\Gamma$ as a map
$$\Gamma:\hspace{.5cm}TM\oplus E\rightarrow\Lambda^2(TM\oplus E)\subset Cl(TM\oplus E),$$
and compute the term $\mathcal{C}_3$ in (\ref{R function B}). According to Lemma \ref{lem1 ap1}, for all $X,Y\in TM\oplus E,$ 
$$\Gamma(X)(Y)=\left[\Gamma(X),Y\right].$$
\begin{lem}
If $X,Y\in TM,$
\begin{eqnarray*}
\frac{1}{2}\left((\nabla_X\Gamma)(Y)-(\nabla_Y\Gamma)(X)\right)&=&[\Gamma(X),\Gamma(Y)]-\frac{1}{2}\Gamma([\Gamma(X),Y]-[\Gamma(Y),X])\\
&&-\frac{1}{2}\left[\sum_{j=1}^pe_j\cdot B(X,e_j),\Gamma(Y)\right]+\frac{1}{2}\left[\sum_{j=1}^pe_j\cdot B(Y,e_j),\Gamma(X)\right].
\end{eqnarray*}
Here the brackets stand for the commutator in $Cl(TM\oplus E).$ 
\end{lem}
\begin{proof}
By Lemmas \ref{lem1 ap1} and \ref{lem2 ap1} in the appendix, the linear maps $\Gamma(X)\circ\Gamma(Y)-\Gamma(Y)\circ\Gamma(X),$ $Z\mapsto\Gamma(\Gamma(X)Y)Z$ and $Z\mapsto \Gamma(B(X,Y))Z$ appearing in Lemma \ref{lemma formula nabla gamma} are respectively represented by the bivectors  $[\Gamma(X),\Gamma(Y)],$ $\Gamma([\Gamma(X),Y])$ and $\Gamma(B(X,Y)).$ Moreover, by Lemma \ref{lem4 ap1} applied to the linear maps $B(X,.):TM\rightarrow E$ and $\Gamma(Y):TM\oplus E\rightarrow TM\oplus E,$ the map
$$Z\mapsto -B^{*}(X,(\Gamma(Y)Z)^{N})+\Gamma(Y)(B^{*}(X,Z^N))+B(X,(\Gamma(Y)Z)^{T})-\Gamma(Y)(B(X,Z^T))$$
is represented by the bivector
$$\left[\sum_{j=1}^pe_j\cdot B(X,e_j),\Gamma(Y)\right]\hspace{.3cm}\in Cl(TM\oplus E).$$
The result follows.
\end{proof}
We readily deduce the sum of the last four terms in (\ref{R function B}):
\begin{lem}\label{computation other terms}
Let us set, for $X,Y\in TM,$
$$R^G(X,Y)=[\Gamma(X),\Gamma(Y)]-\Gamma\left\{[\Gamma(X),Y]-[\Gamma(Y),X]\right\}\ \in\Lambda^2(TM\oplus E),$$
the curvature tensor of $G,$ pulled-back to $TM\oplus E$ by the bundle isomorphism $f$ introduced in (\ref{def bundle iso f}).
Then
$$\mathcal{C}_1+\mathcal{C}_2+\mathcal{C}_3+\mathcal{C}_4= \frac{1}{2}R^G(X,Y).$$
\end{lem}
We thus get from (\ref{R function B}) the formula
\begin{eqnarray}
R(X,Y)\varphi&=&-\frac{1}{2}\sum_{j=1}^pe_j\cdot\left(\tilde{\nabla}_X B(Y,e_j)-\tilde{\nabla}_Y B(X,e_j)\right)\cdot\varphi\label{R function B 2}\\
&&+\mathcal{A}\cdot\varphi+\mathcal{B}\cdot\varphi+\frac{1}{2}R^G(X,Y)\cdot\varphi\nonumber
\end{eqnarray}
where $\mathcal{A}$ and $\mathcal{B}$ are computed in Lemma \ref{computation AB} and $R^G$ may be conveniently written in the form
\begin{eqnarray*}
R^G(X,Y)&=&\sum_{1\leq j<k\leq p}\langle R^G(X,Y)(e_j),e_k\rangle e_j\cdot e_k\\
&&+\sum_{j=1}^{p}\sum_{r=1}^{q}\langle R^G(X,Y)(e_j),n_r\rangle e_j\cdot n_r\\
&&+\sum_{1\leq r<s\leq q}\langle R^G(X,Y)(n_s),n_r\rangle n_r\cdot n_s.
\end{eqnarray*}
On the other hand, the curvature of the spinorial connection is given by
\begin{eqnarray}
R(X,Y)\varphi&=&\frac{1}{2}\left(\sum_{1\leq j<k\leq p}\langle R^T(X,Y)(e_j),e_k\rangle\ e_j\cdot e_k\right.\label{R function RT RN}\\
&&+\left.\sum_{1\leq r<s\leq q}\langle R^N(X,Y)(n_r),n_s\rangle\ n_r\cdot n_s\right)\cdot\varphi.\nonumber
\end{eqnarray}
We now compare the expressions (\ref{R function B 2}) and (\ref{R function RT RN}): since in a given frame $\tilde{s}$ belonging to $\tilde{Q},$ $\varphi$ is represented by an element which is invertible in $Cl(\mathcal{G})$ (it is in fact represented by an element belonging to $Spin(\mathcal{G})$), we may identify the coefficients and get
$$\langle R^T(X,Y)(ej),e_k\rangle =\langle B^*(X,B(Y,e_j)),e_k\rangle- \langle B^*(Y,B(X,e_j)),e_k\rangle+\langle R^G(X,Y)(e_j),e_k\rangle,$$
$$\langle R^N(X,Y)(n_r),n_s\rangle=\langle B(X,B^*(Y,n_r)),n_s\rangle-\langle B(Y,B^*(X,n_r)),n_s\rangle+\langle R^G(X,Y)(n_r),n_s\rangle$$
and
$$\langle\tilde{\nabla}_X B(Y,e_j)-\tilde{\nabla}_Y B(X,e_j),n_r\rangle=\langle R^G(X,Y)(e_j),n_r\rangle$$
for all the indices. These equations are the equations of Gauss, Ricci and Codazzi. 
\\

We now prove that the equations of Gauss, Ricci and Codazzi are also sufficient to get a solution of (\ref{killing equation}). The calculations above in fact show that the connection on $\Sigma$ defined by
\begin{equation}\label{def nabla prime}
\nabla'_X\varphi:=\nabla_X\varphi+\frac{1}{2}\sum_{j=1}^pe_j\cdot B(X,e_j)\cdot\varphi-\frac{1}{2}\Gamma(X)\cdot\varphi
\end{equation}
for all $\varphi\in\Gamma(\Sigma)$ and $X\in\Gamma(TM)$ is flat if and only if the equations of Gauss, Ricci and Codazzi hold. But if this connection is flat there exists a solution $\varphi\in \Gamma(U\Sigma)$ of (\ref{killing equation}); this is because $\nabla'$ may be also interpreted as a connection on $U\Sigma$ regarded as a principal bundle (of group $Spin(\mathcal{G}),$ acting on the right): indeed, $\nabla$ defines such a connection (since it comes from a connection on $\tilde{Q}$), and the right hand side term in (\ref{def nabla prime}) defines a linear map
\begin{eqnarray*}
TM&\rightarrow&\chi^{inv}_V(U\Sigma)\\
X&\mapsto&\varphi\mapsto \frac{1}{2}\sum_{j=1}^pe_j\cdot B(X,e_j)\cdot\varphi-\frac{1}{2}\Gamma(X)\cdot\varphi
\end{eqnarray*}
from $TM$ to the vector fields on $U\Sigma$ which are vertical and invariant under the action of the group (these vector fields are of the form $\varphi\mapsto\eta\cdot\varphi,$ $\eta\in\Lambda^2(TM\oplus E)\subset Cl(TM\oplus E)).$ Assuming that the equations of Gauss, Codazzi and Ricci hold, we thus get a solution  $\varphi\in \Gamma(U\Sigma)$ of (\ref{killing equation}).
\end{proof}
The considerations above give a spinorial proof of the Fundamental Theorem of submanifold theory in the metric Lie group $G$ (see \cite{PT} for another proof). We keep the hypotheses and notation of the beginning of Section \ref{section main result}.
\begin{cor} 
We moreover assume that $M$ is simply connected and that $B:TM\times TM\rightarrow E$ satisfies the equations of Gauss, Codazzi and Ricci (\ref{Gauss equation G})-(\ref{Codazzi equation G}). Then there is an isometric immersion of $M$ into $G$ with normal bundle $E$ and second fundamental form $B.$ The immersion is unique up to a rigid motion in $G,$ that is up to a transformation of the form
\begin{eqnarray}
L_{b}\circ\Phi_a:\hspace{.5cm}G&\rightarrow& G\label{rigid motion}\\
g&\mapsto& b\Phi_a(g)\nonumber
\end{eqnarray}
where $a\in Spin(\mathcal{G})$ is such that $Ad(a):\ \mathcal{G}\rightarrow\mathcal{G}$ is an automorphism of Lie algebra, $\Phi_a:G\rightarrow G$ is the group automorphism such that $d(\Phi_a)_e=Ad(a),$ and $b$ belongs to $G.$ 
\end{cor} 
\begin{proof}
The equations of Gauss, Codazzi and Ricci are the integrability conditions of (\ref{killing equation}). We thus get a solution $\varphi\in \Gamma(U\Sigma)$ of (\ref{killing equation}); with such a spinor field at hand, $F=\int\xi$ where $\xi$ is defined in (\ref{def xi}) is the immersion. Finally, a solution of (\ref{killing equation}) is unique up to the right action of an element of $Spin(\mathcal{G})$; the right multiplication of $\varphi$ by $a\in Spin(\mathcal{G})$ and the left multiplication by $b\in G$ in the last integration give also an immersion, if $Ad(a):\mathcal{G}\rightarrow\mathcal{G}$ is moreover an automorphism of Lie algebra. This immersion is obtained from the immersion defined by $\varphi$ by a rigid motion, as described in (\ref{rigid motion}).
\end{proof}
\begin{rem}
In $\R^n,$ a rigid motion as in (\ref{rigid motion}) is a transformation of the form
\begin{eqnarray*}
\R^n&\rightarrow&\R^n\\
x&\mapsto&ax+b,
\end{eqnarray*}
with $a\in SO(n)$ and $b\in\R^n.$
\end{rem}
\section{Special cases}\label{section special cases}
\subsection{Submanifolds in $\R^n$}\label{section Rn}
If the metric Lie group is $\R^n$ with its natural metric, we recover the main result of \cite{BLR2}. We suppose that $M$ is a $p$-dimensional Riemannian manifold, $E\rightarrow M$ a bundle of rank $q,$ with a fibre metric and a compatible connection. We assume that $TM$ and $E$ are oriented and spin with given spin structures, and that $B:TM\times TM\rightarrow E$ is bilinear and symmetric.
\begin{thm}\label{theorem section Rn}\cite{BLR2}
We moreover assume that $M$ is simply connected. The following statements are equivalent:
\begin{enumerate}
\item There exists a section $\varphi\in\Gamma(U\Sigma)$ such that
\begin{equation}\label{killing equation Rn}
\nabla_X\varphi=-\frac{1}{2}\sum_{j=1}^pe_j\cdot B(X,e_j)\cdot\varphi
\end{equation}
for all $X\in TM.$
\item There exists an isometric immersion $F:\ M\rightarrow \R^n$ with normal bundle $E$ and second fundamental form $B.$\\
Moreover, $F=\int\xi$ where $\xi$ is the $\R^n-$valued 1-form defined by
\begin{equation}\label{def xi Rn}
\xi(X):=\langle\langle X\cdot\varphi,\varphi\rangle\rangle
\end{equation}
for all $X\in TM.$
\end{enumerate}
\end{thm}
\begin{proof}
We only prove $(1)\Rightarrow(2).$ This will be a consequence of Theorem \ref{thm main result} if we may define a bundle map $f$ as in (\ref{def bundle iso f}) such that (\ref{nabla field invariant}) holds. We assume that $\varphi$ is a solution of (\ref{killing equation Rn}), and set
\begin{eqnarray*}
f:\hspace{.5cm}TM\oplus E&\rightarrow& M\times\R^n\\
Z&\mapsto& \langle\langle Z\cdot\varphi,\varphi\rangle\rangle.
\end{eqnarray*}
The map $\Gamma$ defined by (\ref{def Gamma TM+E}) is $\Gamma=0.$ We now show that (\ref{nabla field invariant}) is satisfied for every $Z\in\Gamma(TM\oplus E)$ such that $f(Z):M\rightarrow\R^n$ is a constant map: for all $X\in TM,$ we have $\partial_X \{f(Z)\}=0,$ which reads
$$\langle\langle\nabla_XZ\cdot\varphi,\varphi\rangle\rangle+\langle\langle Z\cdot\nabla_X\varphi,\varphi\rangle\rangle+\langle\langle Z\cdot\varphi,\nabla_X\varphi\rangle\rangle=0.$$
But (\ref{killing equation Rn}) gives
$$\langle\langle Z\cdot\nabla_X\varphi,\varphi\rangle\rangle+\langle\langle Z\cdot\varphi,\nabla_X\varphi\rangle\rangle=\langle\langle \{B(X,Z^T)-B^*(X,Z^N)\}\cdot\varphi,\varphi\rangle\rangle$$
(see the computations in (\ref{d xi computation}) with $\Gamma=0$). Thus
$$\langle\langle\nabla_XZ\cdot\varphi,\varphi\rangle\rangle=\langle\langle \{-B(X,Z^T)+B^*(X,Z^N)\}\cdot\varphi,\varphi\rangle\rangle$$
and 
$$\nabla_XZ=-B(X,Z^T)+B^*(X,Z^N),$$
which is (\ref{nabla field invariant}) with $\Gamma=0.$
\end{proof}
\subsection{Submanifolds in $\mathbb{H}^n$}
Spinor representations of submanifolds in $\HH^n$ with its natural metric were already given in \cite{Mo,BLR,BLR2}. We give here another representation using the group structure of $\HH^n,$ with an arbitrary left invariant metric. 
Let us set
$$\HH^n=\{a=(a',a_n)\ \in\R^n:\ a_n>0\},$$
and, for $a\in \HH^n,$ the transformation 
\begin{eqnarray*}
\varphi_{a}:\hspace{.5cm}\R^{n-1}&\rightarrow&\R^{n-1}\\
x&\mapsto& a_n x+a';
\end{eqnarray*}
$\varphi_a$ is an homothety composed by a translation. The homotheties composed by translations naturally form a group under composition, and the bijection
\begin{eqnarray*}
\varphi:\hspace{.5cm}\HH^n&\rightarrow& \{\mbox{homotheties-translations }\R^{n-1}\rightarrow\R^{n-1}\}\\
a&\mapsto& \varphi_a
\end{eqnarray*}
induces a group structure on $\HH^n:$ it is such that
\begin{equation}\label{prod Hn}
ab=(a_nb'+a',a_nb_n)
\end{equation}
for all $a,b\in\HH^n;$ the identity element is $e=(0,1)\in\HH^n.$ Let us denote by $(e_1^o,e_2^o,\ldots,e_n^o)$ the canonical basis of $T_e\HH^n=\R^n$ and keep the same letters to denote the corresponding left invariant vector fields on $\HH^n$. The Lie bracket may be easily seen to be given by
$$[e_i^o,e_j^o]=0\hspace{.5cm}\mbox{and}\hspace{.5cm} [e_n^o,e_i^o]=e_i^o$$
for $i,j=1,\ldots,n-1.$ This may also be written in the form
\begin{equation}\label{def bracket Rn}
[X,Y]=l(X)Y-l(Y)X
\end{equation}
for all $X,Y\in \R^n,$ where $l:\R^n\rightarrow\R$ is the linear form such that $l(e_i^o)=0$ if $i\leq n-1$ and $l(e_n^o)=1.$ This property implies that every left invariant metric on $\HH^n$ has constant negative curvature $-|l|^2$ \cite{Mi, MP}. 
\\

We suppose that a left invariant metric $\langle.,.\rangle$ is given on $\HH^n,$ and consider the vector $U_o\in T_e\HH^n$ such that $l(X)=\langle U_o,X\rangle$ for all $X\in T_e\HH^n.$ We have $|U_o|=|l|,$ and, by the Koszul formula (\ref{koszul formula}), 
\begin{equation}\label{expr Gamma Hn}
\Gamma(X)(Y)=-\langle Y,U_o\rangle X+\langle X,Y\rangle U_o
\end{equation}
for all $X,Y\in T_e\HH^n.$ 
\\

We keep the hypotheses made at the beginning of Section \ref{section Rn}. We suppose moreover that $U\in\Gamma(TM\oplus E)$ is given such that $|U|=|l|$ and, for all $X\in TM,$
\begin{equation}\label{nabla U Hn}
\nabla_X U=-|U|^2X+\langle X,U\rangle U-B(X,U^T)+B^*(X,U^N).
\end{equation}
We set, for $X\in TM$ and $Y\in TM\oplus E,$
\begin{equation}\label{def Gamma Hn}
\Gamma(X)(Y)=-\langle Y,U\rangle X+\langle X,Y\rangle U.
\end{equation}
\begin{rem}\label{rem U solution eqn Z}
Equation (\ref{nabla U Hn}) implies the following:
\begin{enumerate}
\item $U$ is a solution of (\ref{nabla field invariant}), with the definition (\ref{def Gamma Hn}) of $\Gamma.$
\item The norm of $U$ is constant, since, by a straightforward computation, 
$$d|U|^2(X)=2\langle \nabla_XU,U\rangle=0$$
for all $X\in TM.$ The additional hypothesis $|U|=|l|$ is thus not very restrictive.
\end{enumerate}
We note that it is not necessary to assume the existence of $U$ solution of (\ref{nabla U Hn}) to get a spinor representation of a submanifold in $\HH^n$ if $\HH^n$ is regarded as the set of unit vectors in Minkowski space $\R^{n,1}$ \cite{Mo,BLR,BLR2}. Nevertheless, this hypothesis seems necessary if we consider $\HH^n$ as a group, since the group structure introduces an anisotropy: the vector $e_n\in T_e\HH^n$ is indeed a special direction for the group structure. 
\end{rem}
Let us construct the spinor bundles $\Sigma$ and $U\Sigma$ on $M$ as in Section \ref{section twisted spinor bundle} with here $\mathcal{G}=T_e\HH^n.$ 
\begin{thm}
We assume that $M$ is simply connected. The following statements are equivalent:
\begin{enumerate}
\item There exists a spinor field $\varphi\in\Gamma(U\Sigma)$ solution of (\ref{killing equation}) where $\Gamma$ is defined by (\ref{def Gamma Hn}).
\item There exists an isometric immersion $M\rightarrow \HH^n$  with normal bundle $E$ and second fundamental form $B.$ 
\end{enumerate}
\end{thm}
\begin{proof}
We assume that $\varphi\in\Gamma(U\Sigma)$ is a solution of (\ref{killing equation}) where $\Gamma$ is defined by (\ref{def Gamma Hn}), and define $f:TM\oplus E\rightarrow M\times T_e\HH^n$ by 
$$f(Z)=\langle\langle Z\cdot\varphi,\varphi\rangle\rangle$$
for all $Z\in TM\oplus E.$ Let us first observe that if $Z$ is a vector field solution of (\ref{nabla field invariant}), then $f(Z)$ is constant: we have, for all $X\in TM,$
$$\partial_Xf(Z)=\langle\langle\nabla_XZ\cdot\varphi,\varphi\rangle\rangle+\langle\langle Z\cdot\nabla_X\varphi,\varphi\rangle\rangle+\langle\langle Z\cdot\varphi,\nabla_X\varphi\rangle\rangle;$$
this is 0, by (\ref{nabla field invariant}), (\ref{killing equation}) and the computation (\ref{d xi computation}). Since $U$ is a solution of (\ref{nabla field invariant}) (see Remark \ref{rem U solution eqn Z}), we deduce that $f(U)\in T_e\HH^n$ is a constant, and, since $|f(U)|=|U|=|U_o|,$ replacing $\varphi$ by $\varphi\cdot a$ for some $a\in Spin(T_e\HH^n)$ if necessary, we may suppose that $f(U)=U_o.$ Since $\Gamma$ is defined on $T_e\HH^n$ by (\ref{expr Gamma Hn}) and on $TM\oplus E$ by (\ref{def Gamma Hn}), and since $f$ preserves the metrics, it is straightforward to see that $f(\Gamma(X)(Y))=\Gamma(f(X))(f(Y))$ for all $X,Y\in TM\oplus E.$ Finally, (\ref{nabla field invariant}) holds for all $Z\in\Gamma(TM\oplus E)$ such that $f(Z)$ is constant: this is the same argument as in the proof of Theorem \ref{theorem section Rn} in Section \ref{section Rn}, just adding the term $\Gamma.$ The result then follows from Theorem \ref{thm main result}.
\end{proof}

\subsection{Hypersurfaces in a metric Lie group} 
We assume that $G$ is a simply connected $n$-dimensional metric Lie group, $M$ is a $p$-dimensional Riemannian  manifold,  $n=p+1,$ and $E$ is the trivial line bundle on $M$, oriented by a unit section $\nu\in \Gamma(E).$ We moreover suppose that $M$ is simply connected and that $h:TM\times TM\rightarrow \R$ is a given symmetric bilinear form, and that the hypotheses (1) and (2) of Section \ref{section main result} with $B=h\nu$ hold. According to Theorem \ref{thm main result}, an isometric immersion of $M$ into $G$ with second fundamental form $h$ is equivalent to a section $\varphi$ of $\Gamma(U\Sigma)$ solution of the Killing equation (\ref{killing equation}). Note that $Q_E\simeq M$ and the double covering $\tilde{Q}_E\rightarrow Q_{E}$ is trivial, since $M$ is assumed to be simply connected. Fixing a section $\tilde{s}_E$ of $\tilde{Q}_E$ we get an injective map
\begin{eqnarray*}
\tilde{Q}_M&\rightarrow &\tilde{Q}_M\times_M\tilde{Q}_E=:\tilde{Q}\\
\tilde{s}_M&\mapsto&(\tilde{s}_M,\tilde{s}_E).
\end{eqnarray*}
Using
$$Cl_p\simeq Cl^0_{p+1}\subset Cl_{p+1}$$
(induced by the Clifford map $\R^p\rightarrow Cl_{p+1},$ $X\mapsto X\cdot e_{p+1})$,  we deduce a bundle isomorphism
\begin{eqnarray}
\tilde{Q}_M\times_{\rho} Cl_p&\rightarrow& \tilde{Q}\times_{\rho} Cl^0_{p+1}\hspace{.3cm} \subset\Sigma\label{identif spineurs}\\
\psi&\mapsto&\psi^*.\nonumber
\end{eqnarray}
It satisfies the following properties: for all $X\in TM$ and $\psi\in \tilde{Q}_M\times_{\rho} Cl_p,$
\begin{equation}\label{properties spinors M G}(X\cdot\psi)^*=X\cdot\nu\cdot\psi^*
\hspace{.5cm}\mbox{and}\hspace{.5cm}\nabla_X(\psi^*)=(\nabla_X\psi)^*.
\end{equation}
To write down the Killing equation (\ref{killing equation}) in the bundle $\tilde{Q}_M\times_{\rho} Cl_p,$ we need to decompose the Clifford action of $\Gamma(X)$ into its tangent and its normal parts:
\begin{lem} \label{lemma Gamma Ti} Recall the notation introduced in Remark \ref{rmk cond frame}. Then, for all $X\in TM,$
\begin{equation}
\Gamma(X)=\sum_{i=1}^n\langle X,T_i\rangle\sum_{1\leq j<k\leq n}\Gamma_{ij}^k\left(\frac{1}{2}(T_j\cdot T_k-T_k\cdot T_j)+(f_kT_j-f_jT_k)\cdot\nu\right).
\end{equation}
\end{lem}
\begin{proof}
We have
$$X=\sum_{i=1}^n\langle X,\underline{e}_i\rangle \underline{e}_i=\sum_{i=1}^n\langle X,T_i\rangle \underline{e}_i,$$
\begin{eqnarray*}
\Gamma(X)(\underline{e}_j)&=&\sum_{i=1}^n\langle X,T_i\rangle \Gamma(\underline{e}_i)(\underline{e}_j)\\
&=&\sum_{i=1}^n\langle X,T_i\rangle\sum_{k=1}^n\Gamma_{ij}^k\underline{e}_k\\
&=&\sum_{1\leq i,k\leq n}\Gamma_{ij}^k\langle X,T_i\rangle(T_k+f_k\nu),
\end{eqnarray*}
and thus
\begin{eqnarray*}
\Gamma(X)&=&\frac{1}{2}\sum_{j=1}^n\underline{e}_j\cdot\Gamma(X)(\underline{e}_j)\\
&=&\frac{1}{2}\sum_{j=1}^n(T_j+f_j\nu)\cdot\sum_{1\leq i,k\leq n}\Gamma_{ij}^k\langle X,T_i\rangle(T_k+f_k\nu)\\
&=&\frac{1}{2}\sum_{1\leq i,j,k\leq n}\Gamma_{ij}^k\langle X,T_i\rangle(T_j+f_j\nu)\cdot (T_k+f_k\nu).
\end{eqnarray*}
Now
$$(T_j+f_j\nu)\cdot (T_k+f_k\nu)=T_j\cdot T_k+f_kT_j\cdot\nu-f_jT_k\cdot\nu-f_jf_k,$$
and the result follows since $\Gamma_{ij}^k=-\Gamma_{ik}^j.$
\end{proof}
The section $\varphi\in \Gamma(U\Sigma)$ solution of (\ref{killing equation}) thus identifies to a section $\psi$ of $\tilde{Q}_M\times_{\rho} Cl_p$ solution of
\begin{eqnarray*}
\nabla_X\psi&=&-\frac{1}{2}\sum_{j=1}^ph(X,e_j)e_j\cdot\psi+\frac{1}{2}\tilde{\Gamma}(X)\cdot\psi\\
&=&-\frac{1}{2}S(X)\cdot\psi+\frac{1}{2}\tilde{\Gamma}(X)\cdot\psi
\end{eqnarray*}
for all $X\in TM,$ where 
\begin{equation}\label{def Gamma tilde}
\tilde{\Gamma}(X)=\sum_{i=1}^n\langle X,T_i\rangle\sum_{1\leq j<k\leq n}\Gamma_{ij}^k\left(\frac{1}{2}\left(T_j\cdot T_k-T_k\cdot T_j\right)+(f_kT_j-f_jT_k)\right)
\end{equation}
and $S:TM\rightarrow TM$ is the symmetric operator associated to $h.$ We deduce the following result:
\begin{thm}\label{thm hypersurfaces}
Let $S:TM\rightarrow TM$ be a symmetric operator. The following two statements are equivalent:
\begin{enumerate}
\item there exists an isometric immersion of $M$ into $G$ with shape operator $S;$
\item there exists a normalized spinor field $\psi\in \Gamma(\tilde Q_M\times_{\rho} Cl_p)$ solution of \begin{equation}\label{equation psi}
\nabla_X\psi=-\frac{1}{2}S(X)\cdot\psi+\frac{1}{2}\tilde{\Gamma}(X)\cdot\psi
\end{equation}
for all $X\in TM,$ where $\tilde{\Gamma}$ is defined in (\ref{def Gamma tilde}). 
\end{enumerate}
Here, a spinor field $\psi\in \Gamma(\tilde Q_M\times_{\rho} Cl_p)$ is said to be normalized if it is represented in some frame $\tilde{s}\in\tilde{Q}_M$ by an element $[\psi]\in Cl_p\simeq Cl_{p+1}^0$ belonging to $Spin(p+1).$
\end{thm}
We will see below explicit representation formulas in the cases of the dimensions 3 and 4.

\subsection{Surfaces in a 3-dimensional metric Lie group} 
Since $Cl_2\simeq\Sigma_2$ we have
$$\tilde{Q}_M\times_{\rho} Cl_2\simeq\Sigma M,$$
and $\varphi$ is equivalent to a spinor field $\psi\in\Gamma(\Sigma M)$ solution of (\ref{equation psi}) and such that $|\psi|=1.$ Moreover, the explicit representation formula $F=\int\xi$ may be written in terms of $\psi:$ it may be proved by a computation that
\begin{equation}\label{explicit representation dim 3}
\langle\langle X\cdot\varphi,\varphi\rangle\rangle=i2\mathcal{R}e\langle X\cdot\psi^+,\psi^-\rangle+j\left(\langle X\cdot\psi^+,\alpha(\psi^+)\rangle-\langle X\cdot\psi^-,\alpha(\psi^-)\rangle\right)
\end{equation}
where the brackets $\langle.,.\rangle$ stand here for the natural hermitian product on $\Sigma_2$ and $\alpha:\Sigma_2\rightarrow\Sigma_2$ is the natural quaternionic structure. If $G=\R^3,$ this is the explicit representation formula given in \cite{Fr} (see also \cite{BLR}).
\\

We also note that the expression (\ref{def Gamma tilde}) of $\tilde{\Gamma}$ simplifies if the Lie group is 3-dimensional:
\begin{lem}\label{lem Gamma dim 3}
If $j,k,$ $j\neq k,$ belong to $\{1,2,3\},$ let us denote by $l_{jk}\in \{1,2,3\}$ the number such that $(j,k,l_{jk})$ is a permutation of $\{1,2,3\}$ and by $\epsilon_{jk}=\pm 1$ the sign of this permutation. Then, for all $X\in TM,$
$$\tilde{\Gamma}(X)=\sum_{i=1}^3\langle X,T_i\rangle\sum_{1\leq j<k\leq 3}\Gamma_{ij}^k\epsilon_{jk}(f_{l_{jk}}-T_{l_{jk}})\cdot\omega$$
where $\omega\in Cl(TM)$ is the area element of $M.$
\end{lem}
\begin{proof}
Keeping the notation introduced above, we note that
$$\underline{e}_j\cdot \underline{e}_k\cdot \underline{e}_{l_{jk}}=\epsilon_{jk}\ \omega\cdot\nu,$$
which yields
$$\underline{e}_j\cdot \underline{e}_k=-\epsilon_{jk}\ \omega\cdot\nu\cdot \underline{e}_{l_{jk}}.$$
Thus
\begin{eqnarray*}
T_j\cdot T_k+(f_kT_j-f_jT_k)\cdot\nu-f_jf_k&=&-\epsilon_{jk}\ \omega\cdot\nu\cdot(T_{l_{jk}}+f_{l_{jk}}\nu)\\
&=&\epsilon_{jk}(f_{l_{jk}}-T_{l_{jk}}\cdot\nu)\cdot\omega
\end{eqnarray*}
since $T_{l_{jk}}\cdot\nu=-\nu\cdot T_{l_{jk}},$ $T_{l_{jk}}\cdot\omega=-\omega\cdot T_{l_{jk}}$ and $\omega\cdot\nu=\nu\cdot\omega.$ Switching the indices $j$ and $k$ we also get
\begin{eqnarray*}
T_k\cdot T_j+(f_jT_k-f_kT_j)\cdot\nu-f_kf_j&=&\epsilon_{kj}(f_{l_{kj}}-T_{l_{kj}}\cdot\nu)\cdot\omega\\
&=&-\epsilon_{jk}(f_{l_{jk}}-T_{l_{jk}}\cdot\nu)\cdot\omega
\end{eqnarray*}
since  $\epsilon_{kj}=-\epsilon_{jk}$ and $l_{kj}=l_{jk}.$ We deduce that
$$\frac{1}{2}\left(T_j\cdot T_k-T_k\cdot T_j\right)+(f_kT_j-f_jT_k)\cdot\nu=\epsilon_{jk}(f_{l_{jk}}-T_{l_{jk}}\cdot\nu)\cdot\omega.$$
The result is then a consequence of Lemma \ref{lemma Gamma Ti} together with the relation 
$$\left(\tilde{\Gamma}(X)\cdot \psi\right)^* =\Gamma(X)\cdot \psi^*$$ 
and the first property in (\ref{properties spinors M G}). 
\end{proof}
\subsubsection{The metric Lie group $S^3$}
A spinor representation of a surface immersed in $S^3$ was already given in \cite{Mo} (see also \cite{BLR,BLR2}). We give here a spinor representation relying on the group structure; it appears that it coincides with the result in \cite{Mo}.
\\

We regard the sphere $S^3$ as the set of the unit quaternions, with its natural group structure. The Lie algebra of $S^3$ identifies to $\R^3,$ with the bracket $[X,Y]=2X\times Y$ for all $X,Y\in \R^3$ ($\times$ is the usual cross product). By the Koszul formula (\ref{koszul formula}), for all $X,Y\in\R^3,$
$$\Gamma(X)(Y)=X\times Y.$$
As a bivector, for all $X=X_1e_1^o+X_2e_2^o+X_3e_3^o\in\R^3,$
\begin{eqnarray*}
\Gamma(X)&=&\frac{1}{2}\left(e_1^o\cdot \Gamma(X)(e_1^o)+e_2^o\cdot \Gamma(X)(e_2^o)+e_3^o\cdot \Gamma(X)(e_3^o)\right)\\
&=&X_1e_2^o\cdot e_3^o+X_2e_3^o\cdot e_1^o+X_3e_1^o\cdot e_2^o\\
&=&-X\cdot(e_1^o\cdot e_2^o\cdot e_3^o).
\end{eqnarray*}
Thus, if $\varphi\in \tilde{Q}\times_{\rho} Cl_3^0$ represents an immersion of an oriented surface $M$ in $S^3$ and if $\psi\in\Gamma(\Sigma M)$ is such that $\varphi=\psi^*,$ then, for all $X\in TM,$
\begin{eqnarray*}
\Gamma(X)\cdot\varphi&=&-X\cdot(e_1^o\cdot e_2^o\cdot e_3^o)\cdot\varphi\\
&=&-X\cdot\omega\cdot\nu\cdot\varphi\\
&=&(X\cdot\nu)\cdot\omega\cdot\varphi\\
&=&\left(X\cdot\omega\cdot\psi\right)^*
\end{eqnarray*}
where $\omega$ is the area form of $M$, and $\nu$ is the vector normal to $M$ in $S^3.$
Since $\varphi\in\Gamma(U\Sigma)$ is a solution of (\ref{killing equation}), $\psi\in\Gamma(\Sigma M)$ is a solution of
$$\nabla_X\psi=-\frac{1}{2}S(X)\cdot\psi+\frac{1}{2}X\cdot\omega\cdot\psi$$
and satisfies $|\psi|=1.$ Taking the trace, we get
\begin{eqnarray*}
D\psi&=&e_1\cdot\nabla_{e_1}\psi+e_2\cdot\nabla_{e_2}\psi\\
&=&H\psi-\omega\cdot\psi
\end{eqnarray*}
where $(e_1,e_2)$ is a positively oriented and orthonormal basis of $TM.$ Now, setting $\overline{\psi}=\psi^+-\psi^-$ and since $\omega\cdot\psi=-i\overline{\psi}$ (recall that $i\omega$ acts as the identity on $\Sigma^+M$ and as -identity on $\Sigma^-M$),  we get 
$$D\psi=H\psi-i\overline{\psi},$$
which is also the spinor characterization given by Morel in \cite{Mo}.

\subsubsection{Surfaces in the 3-dimensional metric Lie groups $E(\kappa,\tau),$ $\tau\neq 0$}
We recover here a spinor characterization of immersions in the 3-dimensional homogeneous spaces $E(\kappa,\tau);$ this result was obtained by the second author in \cite{Ro}, using a characterization of immersions in these spaces by Daniel \cite{Da}. We give here an independent proof, and rather obtain the result of Daniel as a corollary.
\\

The metric Lie group $E(\kappa,\tau),$ $\tau\neq 0,$ is defined as follows: its Lie algebra is $\mathcal{G}=\R^3,$ with the bracket defined on the vectors $e_1^o,e_2^o,e_3^o$ of the canonical basis by
$$[e_1^o,e_2^o]=2\tau e_3^o,\hspace{1cm} [e_2^o,e_3^o]=\sigma e_1^o,\hspace{1cm} [e_3^o,e_1^o]=\sigma e_2^o$$
where $\sigma=\frac{\kappa}{2\tau}.$ The metric on $\mathcal{G}$ is the canonical metric, ie the metric such that the basis $(e_1^o,e_2^o,e_3^o)$ is orthonormal. The Levi-Civita connection is then given by
\begin{equation}\label{Gamma E k t}
\Gamma(X)(Y)=\left\{\tau(X-\langle X,e_3^o\rangle e_3^o)+(\sigma-\tau)\langle X,e_3^o\rangle e_3^o\right\}\times Y
\end{equation}
for $X,Y\in\mathcal{G};$ see e.g. \cite{Da}.
\\

Let $S:TM\rightarrow TM$ be a symmetric operator. We assume that a vector field $T\in\Gamma(TM)$ and a function $f\in C^{\infty}(M,\R)$ are given such that
\begin{equation}\label{eqn T f E k t}
|T|^2+f^2=1,
\end{equation}
\begin{equation}\label{eqn T E k t}
\nabla_X T=f(S(X)-\tau JX)
\end{equation}
and
\begin{equation}\label{eqn f E k t}
df(X)=-\langle S(X)-\tau JX,T\rangle
\end{equation}
for all $X\in TM,$ where $J:TM\rightarrow TM$ denotes the rotation of angle $+\pi/2$ in the tangent planes.

\begin{thm}\label{thm E k t}\cite{Ro}
If $M$ is simply connected, the following two statements are equivalent:
\begin{enumerate}
\item There exists $\psi\in\Gamma(\Sigma M)$ such that $|\psi|=1$ and
\begin{equation}\label{eqn spinor E k t}
\nabla_X\psi=-\frac{1}{2}S(X)\cdot\psi+\frac{1}{2}\left\{(2\tau-\sigma)\langle X,T\rangle\left(T-f\right)-\tau X\right\}\cdot\omega\cdot\psi
\end{equation}
for all $X\in TM.$
\item There exists an isometric immersion of $M$ into $E(\kappa,\tau),$ with shape operator $S.$
\end{enumerate}
\end{thm}
\begin{proof}
We consider the trivial line bundle $E=\R\nu,$ where $\nu$ is a unit section. The bundle $TM\oplus E$ is of rank 3, and is assumed to be oriented by the orientation of $TM$ and by $\nu.$ We suppose that it is endowed with the natural product metric. Let us denote by $\times$ the natural cross product in the fibers. We set 
$$\underline{e}_3=T+f\nu,$$
and, for all $X,Y\in TM\oplus E,$
\begin{equation}\label{def Gamma TM+E E k t}
\Gamma(X)(Y)=\left\{\tau(X-\langle X,\underline{e}_3\rangle \underline{e}_3)+(\sigma-\tau)\langle X,\underline{e}_3\rangle \underline{e}_3\right\}\times Y.
\end{equation}
Defining $B:TM\times TM\rightarrow E$ and its adjoint $B^*:TM\times E\rightarrow TM$ by
\begin{equation}\label{def B E k t}
B(X,Y)=\langle S(X),Y\rangle \nu\hspace{1cm}\mbox{and}\hspace{1cm}B^*(X,\nu)=S(X)
\end{equation}
for all $X,Y\in TM,$ the equations (\ref{eqn T E k t}) and (\ref{eqn f E k t}) are equivalent to the single equation 
\begin{equation}\label{eqn e3 invariant}
\nabla_X\underline{e}_3=\Gamma(X)(\underline{e}_3)-B(X,\underline{e}_3^T)+B^*(X,\underline{e}_3^N)
\end{equation}
for all $X\in TM,$ where $\nabla$ is the sum of the Levi-Civita connection on $TM$ and the trivial connection on $E.$ This is (\ref{nabla field invariant}) for $Z=\underline{e}_3.$ We will need the following expression for $\Gamma:$
\begin{lem}\label{lemma Gamma E k t}
For all $X\in TM,$ the linear map $\Gamma(X):TM\oplus E\rightarrow TM\oplus E$ defined by (\ref{def Gamma TM+E E k t})  is represented by the bivector
$$\Gamma(X)=\left\{\left(2\tau-\sigma\right)\langle X,T\rangle\left(T\cdot\nu-f\right)-\tau X\cdot\nu\right\}\cdot\omega.$$
\end{lem}
\begin{proof}
The linear map $\Gamma(X)$ is represented by the bivector
$$\Gamma(X)=\frac{1}{2}\left(\underline{e}_1\cdot\Gamma(X)(\underline{e}_1)+\underline{e}_2\cdot\Gamma(X)(\underline{e}_2)+\underline{e}_3\cdot\Gamma(X)(\underline{e}_3)\right)$$
where $\underline{e}_1,\underline{e}_2$ are such that $\underline{e}_1,\underline{e}_2,\underline{e}_3$ is a positively oriented and orthonormal basis of $TM\oplus E$ (see Lemma \ref{lem1 ap1}); thus, a straightforward computation shows that $\Gamma(X)$ is represented by the bivector
\begin{equation}\label{lem Gamma E k t}
\Gamma(X)=-\tau(X\times \underline{e}_3)\cdot \underline{e}_3+(\sigma-\tau)\langle X,\underline{e}_3\rangle\ \underline{e}_1\cdot \underline{e}_2.
\end{equation}
The following formula may be checked by a direct computation: for $X,Y\in TM\oplus E,$
$$X\times Y=-\left( X\cdot Y+\langle X,Y\rangle\right)\underline{e}_1\cdot \underline{e}_2\cdot \underline{e}_3;$$
this gives
\begin{eqnarray*}
(X\times \underline{e}_3)\cdot \underline{e}_3&=&-\left(X\cdot \underline{e}_3+\langle X,\underline{e}_3\rangle\right)\underline{e}_1\cdot \underline{e}_2\cdot \underline{e}_3\cdot \underline{e}_3\\
&=&\left(X-\langle X,\underline{e}_3\rangle \underline{e}_3\right)\underline{e}_1\cdot \underline{e}_2\cdot \underline{e}_3\\
&=&\left(X-\langle X,T\rangle\left(T+f\nu\right)\right)\cdot\omega\cdot\nu\\
&=& \left(X\cdot\nu-\langle X,T\rangle\left(T\cdot\nu-f\right)\right)\cdot\omega.
\end{eqnarray*}
Moreover, 
\begin{eqnarray*}
\langle X,\underline{e}_3\rangle\ \underline{e}_1\cdot \underline{e}_2&=&\langle X,T\rangle\left(-\underline{e}_1\cdot \underline{e}_2\cdot \underline{e}_3\cdot \underline{e}_3\right)\\
&=&\langle X,T\rangle\left(-\omega\cdot\nu\cdot (T+f\nu)\right)\\
&=&-\langle X,T\rangle \left(T\cdot\nu-f\right)\cdot\omega.
\end{eqnarray*}
Plugging these two formulas in (\ref{lem Gamma E k t}) we get the result.
\end{proof}
\noindent We deduce the following key lemma:
\begin{lem}\label{prop phi equiv psi}
A spinor field $\varphi\in\Gamma(U\Sigma)$ solution of (\ref{killing equation}) is equivalent to a spinor field $\psi\in\Gamma(\Sigma M)$ solution of (\ref{eqn spinor E k t}).
\end{lem}
\begin{proof}
We use the identification $\psi\in\Gamma(\Sigma M)\mapsto \psi^*\in\Gamma(\Sigma)$ described at the beginning of the section; we recall that, for all $X\in TM,$
\begin{equation}\label{properties ident psi phi}
\left(\nabla_X\psi\right)^*=\nabla_X(\psi^*)\hspace{.5cm}\mbox{and}\hspace{.5cm}(X\cdot\psi)^*=X\cdot\nu\cdot(\psi^*).
\end{equation}
Thus, if $\varphi\in\Gamma(U\Sigma)$ is a solution of (\ref{killing equation}) and if $\psi\in\Gamma(\Sigma M)$ is such that $\psi^*=\varphi,$ using (\ref{properties ident psi phi}) together with the formula
$$\sum_{j=1}^pe_j\cdot B(X,e_j)=\sum_{j=1}^pe_j\cdot \langle S(X),e_j\rangle\nu=S(X)\cdot\nu$$
and Lemma \ref{lemma Gamma E k t}, we get:
\begin{eqnarray*}
(\nabla_X\psi)^*&=&\nabla_X\varphi\\
&=&-\frac{1}{2}S(X)\cdot\nu\cdot\varphi+\frac{1}{2}\left\{\left(2\tau-\sigma\right)\langle X,T\rangle\left(T\cdot\nu-f\right)-\tau X\cdot\nu\right\}\cdot\omega\cdot\varphi\\
&=&\left(-\frac{1}{2}S(X)\cdot\psi+\frac{1}{2}\left\{\left(2\tau-\sigma\right)\langle X,T\rangle\left(T-f\right)-\tau X\right\}\cdot\omega\cdot\psi\right)^*.
\end{eqnarray*}
This gives (\ref{eqn spinor E k t}). Reciprocally, if $\psi$ is a solution of (\ref{eqn spinor E k t}), the spinor field $\varphi=\psi^*$ satisfies (\ref{killing equation}). This proves the lemma.
\end{proof}
\noindent Instead of $\psi\in\Gamma(\Sigma M)$ solution of (\ref{eqn spinor E k t}) we may thus consider $\varphi\in\Gamma(U\Sigma)$ solution of (\ref{killing equation}). Theorem \ref{thm E k t} will thus be a consequence of Theorem \ref{thm main result} if we can define a bundle isomorphism $f:TM\oplus E\rightarrow M\times\mathcal{G}$ such that (\ref{def Gamma TM+E}) and (\ref{nabla field invariant}) hold. Let us set 
$$f(Z)=\langle\langle Z\cdot\varphi,\varphi\rangle\rangle.$$
We first observe that $f(\underline{e}_3)$ is constant: indeed, for all $X\in TM,$
\begin{equation*}
\partial_X(f(\underline{e}_3))=\langle\langle \nabla_X\underline{e}_3\cdot\varphi,\varphi\rangle\rangle+\langle\langle \underline{e}_3\cdot\nabla_X\varphi,\varphi\rangle\rangle+\langle\langle \underline{e}_3\cdot\varphi,\nabla_X\varphi\rangle\rangle=0
\end{equation*}
in view of (\ref{eqn e3 invariant}), (\ref{killing equation}) and the computation in (\ref{d xi computation}). Moreover, since $f$ preserves the norm of the vectors, $f(\underline{e}_3)$ is a unit vector. Replacing $\varphi$ by $\varphi\cdot a$ for some $a\in Spin(\mathcal{G})$ if necessary, we may thus assume that $f(\underline{e}_3)=e_3^o.$ We now check (\ref{def Gamma TM+E}): since the map $f$ is an orientation preserving isometry and using $f(\underline{e}_3)=e_3^o,$ we have, for all $X,Y\in TM,$
\begin{eqnarray*}
f(\Gamma(X)(Y))&=&f\left(\left\{\tau(X-\langle X,\underline{e}_3\rangle \underline{e}_3)+(\sigma-\tau)\langle X,\underline{e}_3\rangle \underline{e}_3\right\}\times Y\right)\\
&=&\left\{\tau(f(X)-\langle f(X),f(\underline{e}_3)\rangle f(\underline{e}_3))+(\sigma-\tau)\langle f(X),f(\underline{e}_3)\rangle f(\underline{e}_3)\right\}\times f(Y)\\
&=&\left\{\tau(f(X)-\langle f(X),e_3^o\rangle e_3^o)+(\sigma-\tau)\langle f(X),e_3^o\rangle e_3^o\right\}\times f(Y)\\
&=&\Gamma(f(X))(f(Y)).
\end{eqnarray*}
Finally, the proof of (\ref{nabla field invariant}) is very similar to the proof of this identity made in Section \ref{section Rn} for $G=\R^n:$ we only have to add the term involving $\Gamma$ which appears in the expression (\ref{killing equation}) of the covariant derivative of $\varphi$; we leave the details to the reader.
\end{proof}
\begin{rem}
We also get an explicit representation formula: the immersion is given by the Darboux integral of $\xi:$ $X\mapsto\langle\langle X\cdot\varphi,\varphi\rangle\rangle,$ which may be written in terms of $\psi$ by the formula (\ref{explicit representation dim 3}).
\end{rem}
We deduce the following result, first obtained by Daniel in \cite{Da} using the moving frame method:
\begin{cor}
If $S,$ $T,$ $f,$ $\kappa$ and $\tau$ satisfy (\ref{eqn T f E k t})-(\ref{eqn f E k t}), the Gauss equation 
\begin{equation}\label{Gauss equation E k t}
K=\det S+\tau^2+\left(\kappa-4\tau^2\right)f^2
\end{equation}
and the Codazzi equation
\begin{equation}\label{Codazzi equation E k t}
\nabla_X(SY)-\nabla_Y(SX)-S([X,Y])=(\kappa-4\tau^2)f(\langle Y,T\rangle X-\langle X,T\rangle Y),
\end{equation}
then there exists an isometric immersion of $M$ into $E(\kappa,\tau)$ with shape operator $S.$ Moreover the immersion is unique up to a global isometry of $E(\kappa,\tau)$ preserving the orientations.
\end{cor}
\begin{proof}
The equations (\ref{Gauss equation E k t}) and (\ref{Codazzi equation E k t}) are equivalent to the Gauss and Codazzi equations (\ref{Gauss equation G}) and (\ref{Codazzi equation G}) where $B$ is defined by (\ref{def B E k t}). They are thus exactly the integrability conditions for (\ref{killing equation}), and consequently also for (\ref{eqn spinor E k t}).
\end{proof}

\subsubsection{Three-dimensional semi-direct products}
We consider here a semi-direct product $\R^2\rtimes_A\R$ with
$$A=\left(\begin{array}{cc}a&b\\c&d\end{array}\right);$$
if $(e_1^o,e_2^o,e_3^o)$ stands for the canonical basis of $\mathcal{G}=\R^2\times\R,$ the Lie bracket is given by
$$[e_1^o,e_2^o]=0,\hspace{1cm}[e_3^o,e_1^o]=ae_1^o+ce_2^o,\hspace{1cm} [e_3^o,e_2^o]=be_1^o+de_2^o.$$
We equip $\R^2\rtimes_A\R$ with the left invariant metric such that $(e_1^o,e_2^o,e_3^o)$ is orthonormal. By the Koszul formula, we get
\begin{equation}\label{nabla semi-direct1}
\nabla_{e_1^o}e_1^o=a\ e_3^o,\hspace{.5cm}\nabla_{e_1^o}e_2^o=\frac{b+c}{2}\ e_3^o,\hspace{.5cm}\nabla_{e_1^o}e_3^o=-a\ e_1^o-\frac{b+c}{2}\ e_2^o,
\end{equation}
\begin{equation}\label{nabla semi-direct2}
\nabla_{e_2^o}e_1^o=\frac{b+c}{2}\ e_3^o,\hspace{.5cm}\nabla_{e_2^o}e_2^o=d\ e_3^o,\hspace{.5cm}\nabla_{e_2^o}e_3^o=-\frac{b+c}{2}\ e_1^o-d\ e_2^o
\end{equation}
and
\begin{equation}\label{nabla semi-direct3}
\nabla_{e_3^o}e_1^o=\frac{c-b}{2}\ e_2^o,\hspace{.5cm}\nabla_{e_3^o}e_2^o=\frac{b-c}{2}\ e_1^o,\hspace{.5cm}\nabla_{e_3^o}e_3^o=0,
\end{equation}
and deduce
$$\Gamma(X)=\left(aX_1+\frac{b+c}{2} X_2\right)e_1^o\cdot e_3^o+\left(\frac{b+c}{2}X_1+dX_2\right)e_2^o\cdot e_3^o+\frac{c-b}{2}X_3e_1^o\cdot e_2^o$$
for all $X\in\mathcal{G}.$ We first assume that $M$ is an oriented surface in $G=\R^2\rtimes_A\R.$ Recalling that
$$(T_j+f_j\nu)\cdot(T_k+f_k\nu)=\epsilon_{jk}(f_{l_{jk}}-T_{l_{jk}}\cdot\nu)\cdot\omega$$
(see the proof of Lemma \ref{lem Gamma dim 3}), we obtain 
\begin{eqnarray*}
\Gamma(X)&=&-\left(aX_1+\frac{b+c}{2} X_2\right)(f_2-T_2\cdot\nu)\cdot\omega\\
&&+\left(\frac{b+c}{2}X_1+dX_2\right)(f_1-T_1\cdot\nu)\cdot\omega+\frac{c-b}{2}X_3(f_3-T_3\cdot\nu)\cdot\omega
\end{eqnarray*}
and
\begin{eqnarray}
\tilde{\Gamma}(X)&=&-\left(aX_1+\frac{b+c}{2} X_2\right)(f_2-T_2)\cdot\omega\label{tilde Gamma semi-direct}\\
&&+\left(\frac{b+c}{2}X_1+dX_2\right)(f_1-T_1)\cdot\omega+\frac{c-b}{2}X_3(f_3-T_3)\cdot\omega.\nonumber
\end{eqnarray}
Conversely, we consider an oriented Riemannian surface $M,$ and a symmetric operator $S:TM\rightarrow TM.$ We suppose that there exist tangent vectors fields $T_i\in\Gamma(TM)$ and functions $f_i\in C^\infty(M)$ for $1\leqslant i\leqslant 3$ satisfying 
\begin{eqnarray}\label{sol1}
\langle T_i,T_j\rangle +f_if_j=\delta_i^j
\end{eqnarray}
for all $1\leq i,j\leq 3,$ and the equations \eqref{eqn 1 trad} and \eqref{eqn 2 trad} in Remark \ref{rmk cond frame}, with the coefficients $\Gamma_{ij}^k$ given by (\ref{nabla semi-direct1})-(\ref{nabla semi-direct3}). Theorem \ref{thm hypersurfaces} then yields the following result:
\begin{thm}\label{immersion semi-direct} If $M$ is simply connected, the following two statements are equivalent:
\begin{enumerate}
\item there exists an isometric immersion of $M$ into $\R^2\rtimes_A\R$ with shape operator $S$;
\item there exists $\psi\in \Gamma(\Sigma M)$ such that $\vert \psi \vert =1$ and 
\begin{equation}\label{killing equation sol3}
\nabla_X\psi= -\frac{1}{2}S(X)\cdot \psi+\frac{1}{2}\tilde{\Gamma}(X)\cdot\psi
\end{equation}
for all $X\in TM.$
\end{enumerate}
\end{thm}
\noindent
{\bf The metric Lie group $Sol_3$.}
Now, we describe the special case of a surface in $Sol_3:$ this achieves the spinor representation of immersions of surfaces into 3-dimensional Riemannian homogeneous spaces \cite{Ro}.

Let us recall that $Sol_3$ is the only metric Lie group whose isometry group is $3$-dimensional. It is defined as follows: its Lie algebra is $\mathcal{G}=\R^3,$ with the bracket defined on the canonical basis $(e_1^o,e_2^o,e_3^o)$ by 
$$[e_1^o,e_2^o]=0,\hspace{1cm} [e_2^o,e_3^o]=-e_2^o,\hspace{1cm} [e_3^o,e_1^o]=- e_1^o.$$
This is the semi-direct product $\R^2\rtimes_A\R$ with $a=-1,b=c=0,d=1.$ The metric on $\mathcal{G}$ is the canonical metric, i.e., the metric such that the basis $(e_1^o,e_2^o,e_3^o)$ is orthonormal. By the formulas (\ref{nabla semi-direct1})-(\ref{nabla semi-direct3}), the Levi-Civita connection is then such that
\begin{equation}\label{Gamma ijk sol3}
\Gamma_{11}^3=-\Gamma_{13}^1=-1,\hspace{1cm}\Gamma_{22}^3=-\Gamma_{23}^2=1
\end{equation}
and $\Gamma_{ij}^k=0$ for the other indices.  
\\

Let us consider an oriented Riemannian surface $M,$ and a symmetric operator $S:TM\rightarrow TM.$ We suppose that there exist tangent vectors fields $T_i\in\Gamma(TM)$ and functions $f_i\in C^\infty(M)$ for $1\leqslant i\leqslant 3$ satisfying 
\begin{eqnarray}\label{sol1}
\langle T_i,T_j\rangle +f_if_j=\delta_i^j
\end{eqnarray}
for all $1\leq i,j\leq 3,$ and, for all $X\in TM,$
\begin{eqnarray}\label{sol2}
\nabla_{X}T_i=(-1)^{i}\langle X,T_i\rangle T_3+f_iS(X),\\
df_i(X)=(-1)^{i}\langle X,T_i\rangle f_3-\langle SX,T_i\rangle \nonumber
\end{eqnarray}
for $1\leqslant i\leqslant 2,$
\begin{eqnarray}\label{sol3}
\nabla_{X}T_3=\sum_{j=1}^{2}(-1)^{j+1}\langle X,T_j\rangle T_j+f_3S(X),\\
df_3(X)=\sum_{j=1}^{2}(-1)^{j+1}\langle X,T_j\rangle f_j-\langle S(X),T_3\rangle.\nonumber
\end{eqnarray}
The equations \eqref{sol2} and \eqref{sol3} are the equations \eqref{eqn 1 trad} and \eqref{eqn 2 trad} in Remark \ref{rmk cond frame}, with the coefficients $\Gamma_{ij}^k$ given by (\ref{Gamma ijk sol3}). According to (\ref{tilde Gamma semi-direct}) with $a=-1,b=c=0$ and $d=1$ we set
\begin{equation}
\tilde{\Gamma}(X)=\left\{\langle X,T_1\rangle(f_2-T_2)+\langle X,T_2\rangle (f_1-T_1)\right\}\cdot \omega
\end{equation}
for all $X\in TM.$ Theorem \ref{immersion semi-direct} then gives a spinor characterization of an immersion in $Sol_3.$

As a corollary, we obtain a new proof of a result by Lodovici \cite{Lod} concerning existence and uniqueness of isometric immersions in $Sol_3,$ since equation (\ref{killing equation sol3}) is solvable if and only if the equations of Gauss and Codazzi hold (see Section \ref{section fundamental theorem}).\\ \\
{\bf $\HH^2\times\R$ as a metric Lie group.}
Finally, viewing $\HH^2\times\R$ as a metric Lie group, we obtain a new spinor characterization of an immersion in $\HH^2\times\R$ which differs from \cite{Ro} where the product point of view was used.\\
We recall that $\HH^2\times\R$ is the semi-direct product $\R^2\rtimes_A\R$ with $a=1,b=c=d=0.$ The metric on $\mathcal{G}$ is the canonical metric, i.e., the metric such that the basis $(e_1^o,e_2^o,e_3^o)$ is orthonormal. Lie bracket is given by
$$[e_1^o,e_2^o]=0,\hspace{1cm}[e_3^o,e_1^o]=e_1^o,\hspace{1cm} [e_3^o,e_2^o]=0.$$
By the formulas (\ref{nabla semi-direct1})-(\ref{nabla semi-direct3}), the Levi-Civita connection is then such that
\begin{equation}\label{Gamma ijk H2R}
\Gamma_{11}^3=-\Gamma_{13}^1=1
\end{equation}
and $\Gamma_{ij}^k=0$ for the other indices.  \\
Let us consider an oriented Riemannian surface $M,$ and a symmetric operator $S:TM\rightarrow TM.$ We suppose that there exist tangent vectors fields $T_i\in\Gamma(TM)$ and functions $f_i\in C^\infty(M)$ for $1\leqslant i\leqslant 3$ satisfying 
\begin{eqnarray}\label{H2R1}
\langle T_i,T_j\rangle +f_if_j=\delta_i^j
\end{eqnarray}
for all $1\leq i,j\leq 3,$ and, for all $X\in TM,$
\begin{eqnarray}\label{H2RT1}
\nabla_{X}T_1=\langle X,T_1\rangle T_3+f_1S(X),\\
df_1(X)=\langle X,T_1\rangle f_3-\langle SX,T_1\rangle, \nonumber
\end{eqnarray}
\begin{eqnarray}\label{H2RT2}
\nabla_{X}T_2=f_2S(X),\\
df_2(X)=-\langle SX,T_2\rangle, \nonumber
\end{eqnarray}
\begin{eqnarray}\label{H2RT3}
\nabla_{X}T_3=-\langle X,T_3\rangle T_1+f_3S(X),\\
df_1(X)=-\langle X,T_3\rangle f_1-\langle SX,T_3\rangle. \nonumber
\end{eqnarray}
With these identities and according to (\ref{tilde Gamma semi-direct}) with $a=1$,$b=c=d=0$, we set
\begin{equation}
\tilde{\Gamma}(X)=-\langle X,T_1\rangle(f_2-T_2)\cdot \omega
\end{equation}
for all $X\in TM.$ Theorem \ref{immersion semi-direct} then gives a spinor characterization of an immersion in $\HH^2\times\R$.

\subsection{CMC-surfaces in a 3-dimensional metric Lie group}
The aim here is to show that the representation formula for CMC-surfaces in a 3-dimensional metric Lie group by Meeks, Mira, Perez and Ros \cite[Theorem 3.12]{MP} may be obtained as a consequence of the general representation formula in Theorem \ref{thm main result}. For sake of brevity we assume that the group $G$ is unimodular and only give the principal arguments, without details. Under this hypothesis, there exists an orthonormal basis $(e_1^o,e_2^o,e_3^o)$ of the Lie algebra $\mathcal{G}$ and constants $\mu_1,\mu_2,\mu_3\in\R$ such that the Levi-Civita connection of $G$ is given by
$$\Gamma(X)(e_1^o):=\nabla_Xe_1^o=X_3\mu_3e_2^o-X_2\mu_2e_3^o,$$
$$\Gamma(X)(e_2^o):=\nabla_Xe_2^o=-X_3\mu_3e_1^o+X_1\mu_1e_3^o,$$
$$\Gamma(X)(e_3^o):=\nabla_Xe_3^o=X_2\mu_2e_1^o-X_1\mu_1e_2^o$$
(see e.g. \cite[Section 2.6]{MP}), i.e.
\begin{eqnarray}
\Gamma(X)&=&\frac{1}{2}\left(e_1^o\cdot\Gamma(X)(e_1^o)+e_2^o\cdot\Gamma(X)(e_2^o)+e_3^o\cdot\Gamma(X)(e_3^o)\right)\nonumber\\
&=&X_1\mu_1 e_2^o\cdot e_3^o+X_2\mu_2 e_3^o\cdot e_1^o+X_3\mu_3 e_1^o\cdot e_2^o\label{expr Gamma X eio}
\end{eqnarray}
for all $X\in\mathcal{G}.$ Following \cite{MP} we introduce the $H$-potential of the group $G$
\begin{equation}
R(g)=H\left(1+|g|^2\right)^2-\frac{i}{2}\left(\mu_1|1-g^2|^2+\mu_2|1+g^2|^2+4\mu_3|g|^4\right)
\end{equation}
for all $g\in\overline{\mathbb{C}}.$ The importance of this quantity appears in the following lemma, which will permit to express the right-hand side of the Dirac equation (\ref{dirac equation}):
\begin{lem}\label{Gamma CMC}
Let us consider a positively oriented and orthonormal basis $e_1,e_2,\nu$ of $\mathcal{G}$ and set, for $\nu=\nu_1e_1^o+\nu_2e_2^o+\nu_3e_3^o,$
\begin{equation}\label{def T nu}
T(\nu)=\mu_1\nu_1e_1^o+\mu_2\nu_2e_2^o+\mu_3\nu_3e_3^o,
\end{equation}
$A=\frac{1}{2}\langle e_2,T(\nu)\rangle$ and $B=-\frac{1}{2}\langle e_1,T(\nu)\rangle.$ Then, if
$$g=\frac{\nu_1+i\nu_2}{1+\nu_3}$$
is the stereographic projection of $\nu\in S^2$ with respect to the south pole $-e_3^o$ of $S^2$, we have
\begin{eqnarray*}
H\nu+\frac{1}{2}\left(e_1\cdot\Gamma(e_1)+e_2\cdot\Gamma(e_2)\right)&=&\frac{1}{(1+|g|^2)^2}\left(\Re e\ R(g)-\Im m\ R(g)\ e_1\cdot e_2\right)\cdot \nu\\
&&+Ae_1+Be_2.
\end{eqnarray*}
\end{lem}
\begin{proof}
For $i\in\{1,2,3\},$ let us denote by
$$p(e_i^o):=\langle e_i^o,e_1\rangle e_1+\langle e_i^o,e_2\rangle e_2$$
the orthogonal projection of the vector $e_i^o$ onto the plane generated by $e_1$ and $e_2.$ By (\ref{expr Gamma X eio}) we have
\begin{eqnarray*}
e_1\cdot\Gamma(e_1)+e_2\cdot\Gamma(e_2)&=&\mu_1\ p(e_1^o)\cdot e_2^o\cdot e_3^o+\mu_2\ p(e_2^o)\cdot e_3^o\cdot e_1^o+\mu_3\ p(e_3^o)\cdot e_1^o\cdot e_2^o.
\end{eqnarray*}
The proof is then a direct and long computation using that $p(e_i^o)=e_i^o-\langle e_i^o,\nu\rangle\nu$ together with the formulas
\begin{equation}\label{formulas nu g}
\nu_1=\frac{2\ \Re e\ g}{1+|g|^2},\hspace{1cm}\nu_2=\frac{2\ \Im m\ g}{1+|g|^2},\hspace{1cm}\nu_3=\frac{1-|g|^2}{1+|g|^2}.
\end{equation}
\end{proof}
We consider the Clifford map
\begin{eqnarray}
\mathcal{G}&\rightarrow&\HH(2)\label{app Clifford G}\\
x_1e_1^o+x_2e_2^o+x_3e_3^o&\mapsto&\left(\begin{array}{cc}ix_3+ j(x_1-ix_2)&0\\0&-ix_3-j(x_1-ix_2)\end{array}\right)\nonumber
\end{eqnarray}
which identifies $\mathcal{G}$ to the imaginary quaternions so that
\begin{equation}\label{ident basis}
e_1^o\simeq\left(\begin{array}{cc}
j&0\\
0&-j
\end{array}\right)\simeq j,\hspace{.5cm}
e_2^o\simeq\left(\begin{array}{cc}
-ji&0\\
0&ji
\end{array}\right)\simeq -ji,\hspace{.5cm}
e_3^o\simeq\left(\begin{array}{cc}
i&0\\
0&-i
\end{array}\right)\simeq i.
\end{equation}
It identifies $Cl(\mathcal{G})$ to the set of matrices
\begin{equation}\label{id ClG}
\left\{\left(\begin{array}{cc}
a&0\\
0&b
\end{array}\right),\hspace{.5cm} a,b\in\HH\right\}
\end{equation}
and $Spin(\mathcal{G})$ to the group of unit quaternions
$$\left\{\left(\begin{array}{cc}
a&0\\
0&a
\end{array}\right),\hspace{.5cm} a\in\HH,\ |a|=1\right\}\simeq \{a\in\HH,\ |a|=1\}.$$ 
We choose a conformal parameter $z=x+iy$ of the surface, and denote by $\mu$ the real function such that the metric is $\mu^2(dx^2+dy^2).$ In a spinorial frame above the orthonormal frame $e_1=\frac{1}{\mu}\partial_x,$ $e_2=\frac{1}{\mu}\partial_y,$ the spinor field $\varphi$ is represented by $[\varphi]=z_1+jz_2$ where $z_1,z_2\in\C$ are such that $|z_1|^2+|z_2|^2=1.$ 
\begin{lem}
The Dirac equation (\ref{dirac equation}) is equivalent to the system
\begin{eqnarray}
\frac{1}{\sqrt{\mu}}\partial_{\overline{z}}\left(\sqrt{\mu}\ \overline{z_1}\right)&=&i\frac{\mu}{2}\frac{\overline{R(g)}}{(1+|g|^2)^2}\overline{z_2}+\frac{\mu}{2}(A+iB)\overline{z_1}\label{eqn dirac coor 1}\\
\frac{1}{\sqrt{\mu}}\partial_{\overline{z}}\left(\sqrt{\mu}\ z_2\right)&=&-i\frac{\mu}{2}\frac{\overline{R(g)}}{(1+|g|^2)^2}z_1+\frac{\mu}{2}(A+iB)z_2.\label{eqn dirac coor 2}
\end{eqnarray}
Moreover, the $\mathcal{G}-$valued 1-form $\xi$ in Theorem \ref{thm main result} is
\begin{eqnarray}
\xi(X)&=&i\left\{2x_1\ \Im m(z_1\overline{z_2})-2x_2\ \Re e(z_1\overline{z_2})+x_3\left(|z_1|^2-|z_2|^2\right)\right\}\label{expr xi coord}\\
&&+j\left\{x_1(z_1^2+z_2^2)-ix_2(z_1^2-z_2^2)-2ix_3z_1z_2\right\}\nonumber
\end{eqnarray}
for all $X=x_1e_1+x_2e_2+x_3\nu\in TM\oplus E.$
\end{lem}
\begin{proof}
We use here the identification $\psi\in\Gamma(\Sigma M)\mapsto \psi^*\in\Gamma(\Sigma)$ satisfying the properties (\ref{properties ident psi phi}): according to Lemma \ref{Gamma CMC}, the spinor field $\psi\in\Gamma(\Sigma M)$ such that $\psi^*=\varphi$ is solution of
\begin{equation}\label{dirac equation psi MP}
D\psi=\frac{1}{(1+|g|^2)^2}\left(\Re e\ R(g)-\Im m\ R(g)\ e_1\cdot e_2\right)\cdot\psi+(Ae_1+Be_2)\cdot\psi.
\end{equation}
We identify $Cl_2$ to $\HH$ using the Clifford map
\begin{eqnarray}
\R^2&\rightarrow&\HH\label{app Clifford Cl2}\\
(x_1,x_2)&\mapsto& j(x_1-ix_2)\nonumber
\end{eqnarray}
so that, in the fixed spinorial frame above $e_1=\frac{1}{\mu}\partial_x,$ $e_2=\frac{1}{\mu}\partial_y,$
$$[e_1]=j,\hspace{1cm}[e_2]=-ji,\hspace{1cm}[e_1\cdot e_2]=i.$$ 
Using moreover that
$$[\nabla_{\partial x}\psi]=\partial_x[\psi]-\frac{i}{2\mu}\partial_y\mu\ [\psi]\hspace{1cm}[\nabla_{\partial y}\psi]=\partial_y[\psi]+\frac{i}{2\mu}\partial_x\mu\ [\psi]$$
(by (\ref{nabla phi rho}), and the computation of the Christoffel symbols), the left-hand side of (\ref{dirac equation psi MP}) is
$$[D\psi]=\frac{1}{\mu}j\left\{\partial_{x}[\psi]-\frac{i}{2\mu}\partial_y\mu\ [\psi]\right\}-\frac{1}{\mu}ji\left\{\partial_{y}[\psi]+\frac{i}{2\mu}\partial_x\mu\ [\psi]\right\}$$
whereas the right-hand side is
$$\left(\frac{\overline{R(g)}}{(1+|g|^2)^2}+j(A-iB)\right)[\psi].$$
We finally need to precise the identification $\psi\mapsto\psi^*:$  in spinorial frames above $e_1,e_2$ and $e_1,e_2,\nu,$ since the second property in (\ref{properties ident psi phi}) is required and using the Clifford maps (\ref{app Clifford G}) and (\ref{app Clifford Cl2}), it is not difficult to see that the map $\psi\mapsto\psi^*$ corresponds to the map 
$$u+jv\mapsto \left(\begin{array}{cc}u+jiv&0\\0&u+jiv\end{array}\right);$$ 
$\psi$ is thus represented by the quaternion $[\psi]=z_1-jiz_2.$ Direct computations then give the system (\ref{eqn dirac coor 1})-(\ref{eqn dirac coor 2}).

Expression (\ref{expr xi coord}) also follows from a direct computation: we have, in $Cl_3,$ 
\begin{eqnarray*}
\xi(X)&=&\tau[\varphi][X][\varphi]\\
&\simeq&(\overline{z_1}-j\overline{z_2})(ix_3+j(x_1-ix_2))(z_1+jz_2),
\end{eqnarray*}
which easily gives the result.
\end{proof}
We set
\begin{equation}
g=i\frac{\overline{z_2}}{z_1},\hspace{1cm}f=-2\mu z_1^2.
\end{equation}
The function $g$ is the left invariant Gauss map of the surface, stereografically projected with respect to the south pole of $S^2,$ since 
$$\nu=i\left(|z_1|^2-|z_2|^2\right)-2jiz_1z_2$$
is a unit vector normal to the immersion $(x_1=x_2=0$ and $x_3=1$ in (\ref{expr xi coord})) and
$$\frac{\nu_1+i\nu_2}{1+\nu_3}=\frac{2i\ \overline{z_1}\ \overline{z_2}}{1+|z_1|^2-|z_2|^2}=\frac{2i\ \overline{z_1}\ \overline{z_2}}{2|z_1|^2}=i\frac{\overline{z_2}}{z_1}.$$
Direct computations then show that equations (\ref{eqn dirac coor 1})-(\ref{eqn dirac coor 2}) are equivalent to
\begin{equation}\label{equiv dirac 1}
f=4\frac{\partial_zg}{R(g)}
\end{equation}
and
\begin{equation}\label{equiv dirac 2}
\frac{\partial_{\overline{z}}f}{f}=-\frac{2}{1+|g|^2}\partial_{\overline{z}}\overline{g}\ g+\mu (A+iB),
\end{equation}
and that (\ref{expr xi coord}) reads
\begin{equation}\label{xi f g}
\xi=\Re e\left(\ \frac{1}{2}f(\overline{g}^2-1)dz,\ \frac{i}{2}f(\overline{g}^2+1)dz,\ f\overline{g}dz\right)
\end{equation}
in $(e_1^o,e_2^o,e_3^o)$ (recall (\ref{ident basis})). This last formula is the Weierstrass-type representation given in \cite[Theorem 3.15]{MP}. Using that
$$A=\left\langle \xi\left(\frac{\partial_y}{\mu}\right),T(\nu)\right\rangle\hspace{.5cm}\mbox{and}\hspace{.5cm}B=-\left\langle \xi\left(\frac{\partial_x}{\mu}\right),T(\nu)\right\rangle$$
(Lemma \ref{Gamma CMC}) together with (\ref{xi f g}) and (\ref{def T nu}) we get that
\begin{equation}\label{expr A+iB}
A+iB=-\frac{i}{4\mu}\overline{f}\left(\mu_1\nu_1(g^2-1)-i\mu_2\nu_2(g^2+1)+2\mu_3\nu_3g\right).
\end{equation}
Differentiating (\ref{equiv dirac 1}) with respect to $\overline{z}$ and using (\ref{equiv dirac 2}) together with (\ref{expr A+iB}) and (\ref{formulas nu g}) we see by a further computation that $g$ satisfies
\begin{equation}\label{eqn g}
g_{z\overline{z}}=\frac{R_g}{R}g_zg_{\overline{z}}+\left(\frac{R_{\overline{g}}}{R}-\frac{\overline{R_g}}{\overline{R}}\right)|g_z|^2,
\end{equation}
which is the structure equation for the left invariant Gauss map in \cite[Theorem 3.15]{MP}.
\appendix

\section{Skew-symmetric operators and bivectors}
We consider $\R^n$ endowed with its canonical scalar product. A skew-symmetric operator $u:\R^n\rightarrow\R^n$ naturally identifies to a bivector $\underline{u}\in\Lambda^2\R^n,$ which may in turn be regarded as belonging to the Clifford algebra $Cl_n(\R).$ We precise here the relations between the Clifford product in $Cl_n(\R)$ and the composition of endomorphisms.  If $a$ and $b$ belong to the Clifford algebra $Cl_n(\R),$ we set 
$$[a,b]=\frac{1}{2}\left(a\cdot b-b\cdot a\right),$$ 
where the dot $\cdot$ is the Clifford product. We denote by $(e_1,\ldots,e_n)$ the canonical basis of $\R^n.$
\begin{lem} \label{lem1 ap1}
Let $u:\R^n\rightarrow\R^n$ be a skew-symmetric operator. Then the bivector 
\begin{equation}\label{biv rep u}
\underline{u}=\frac{1}{2}\sum_{j=1}^ne_j\cdot u(e_j)\hspace{.3cm}\in\ \Lambda^2\R^n\subset Cl_n(\R)
\end{equation}
represents $u,$ and, for all $\xi\in\R^n,$
$$[\underline{u},\xi]=u(\xi).$$
In the paper, and for sake of simplicity, we will use the same letter $u$ to denote $\underline{u}.$
\end{lem}
\begin{proof}
For $i<j,$ we consider  the linear map 
$$u:\hspace{.5cm}e_i\mapsto e_j,\hspace{.5cm}e_j\mapsto -e_i,\hspace{.5cm} e_k\mapsto 0\hspace{.3cm}\mbox{if}\hspace{.3cm} k\neq i,j;$$
it is skew-symmetric and corresponds to the bivector $e_i\wedge e_j\in\Lambda^2\R^n;$ it is thus naturally represented by  $\underline{u}=e_i\cdot e_j=\frac{1}{2}\left(e_i\cdot e_j-e_j\cdot e_i\right),$ which is (\ref{biv rep u}). We then compute, for $k=1,\ldots,n,$
$$[\underline{u},e_k]=\frac{1}{2}\left(e_i\cdot e_j\cdot e_k-e_k\cdot e_i\cdot e_j\right)$$
and easily get
$$[\underline{u},e_k]=e_j\hspace{.3cm}\mbox{if}\hspace{.3cm}k=i,\hspace{.3cm}-e_i\hspace{.3cm}\mbox{if}\hspace{.3cm}k=j,\hspace{.3cm}0\hspace{.3cm}\mbox{if}\hspace{.3cm}k\neq i,j.$$
The result follows by linearity.
\end{proof}
\begin{lem}\label{lem2 ap1}
Let $u:\R^n\rightarrow\R^n$ and $v:\R^n\rightarrow\R^n$ be two skew-symmetric operators, represented in $Cl_n(\R)$ by 
$$u=\frac{1}{2}\sum_{j=1}^ne_j\cdot u(e_j)\hspace{.5cm}\mbox{and}\hspace{.5cm}v=\frac{1}{2}\sum_{j=1}^ne_j\cdot v(e_j)$$
respectively. Then $[u,v]\in\Lambda^2\R^n\subset Cl_n(\R)$ represents $u\circ v-v\circ u.$
\end{lem}
\begin{proof}
For $\xi\in\R^n,$ the Jacobi equation yields
$$[[u,v],\xi]=[u,[v,\xi]]-[v,[u,\xi]].$$
Thus, using Lemma \ref{lem1 ap1} repeatedly, $[u,v]$ represents the map 
\begin{eqnarray*}
\xi&\mapsto&[[u,v],\xi]
\begin{array}[t]{cc}=&[u,[v,\xi]]-[v,[u,\xi]]\\
=&[u,v(\xi)]-[v,u(\xi)]\\
=&(u\circ v-v\circ u)(\xi),
\end{array}
\end{eqnarray*}
and the result follows.
\end{proof}

We now assume that $\R^n=\R^p\oplus\R^q,$ $p+q=n.$
\begin{lem}\label{lem3 ap1}
Let us consider a linear map $u:\R^p\rightarrow\R^q$ and its adjoint $u^*:\R^q\rightarrow\R^p.$ Then the bivector
$$\underline{u}=\sum_{j=1}^pe_j\cdot u(e_j)\hspace{.3cm}\in\ \Lambda^2\R^n\subset Cl_n(\R)$$
represents 
$$\left(\begin{array}{cc}0&-u^*\\u&0\end{array}\right):\hspace{.5cm}\R^p\oplus\R^q\rightarrow\R^p\oplus\R^q,$$
we have
\begin{equation}\label{biv u u*}
\underline{u}=\frac{1}{2}\left(\sum_{j=1}^pe_j\cdot u(e_j)+\sum_{j=p+1}^ne_j\cdot (-u^*(e_j))\right)
\end{equation}
and, for all $\xi=\xi_p+\xi_q\in\R^n,$
$$[\underline{u},\xi]=u(\xi_p)-u^*(\xi_q).$$
As above, we will simply denote $\underline{u}$ by $u.$
\end{lem}
\begin{proof} In view of Lemma \ref{lem1 ap1}, $\underline{u}$ represents the linear map $\xi\mapsto[\underline{u},\xi].$ We compute, for $\xi\in\R^p,$
\begin{eqnarray*}
[\underline{u},\xi]&=&\frac{1}{2}\left(\sum_{j=1}^pe_j\cdot u(e_j)\cdot\xi-\xi\cdot\sum_{j=1}^pe_j\cdot u(e_j)\right)\\
&=&-\frac{1}{2}\sum_{j=1}^p(e_j\cdot\xi+\xi\cdot e_j)\cdot u(e_j)\\
&=&\sum_{j=1}^p\langle \xi,e_j\rangle\ u(e_j)\\
&=&u(\xi),
\end{eqnarray*}
and, for $\xi\in\R^q,$
\begin{eqnarray*}
[\underline{u},\xi]&=&\frac{1}{2}\left(\sum_{j=1}^pe_j\cdot u(e_j)\cdot\xi-\xi\cdot\sum_{j=1}^pe_j\cdot u(e_j)\right)\\
&=&\frac{1}{2}\sum_{j=1}^pe_j\cdot\left(u(e_j)\cdot\xi+\xi\cdot u(e_j)\right)\\
&=&-\sum_{j=1}^pe_j\ \langle u(e_j),\xi\rangle\\
&=&-\sum_{j=1}^pe_j\ \langle e_j,u^*(\xi)\rangle\\
&=&-u^*(\xi).
\end{eqnarray*}
Finally, 
$$\underline{u}=\sum_{j=1}^pe_j\cdot u(e_j)=\frac{1}{2}\left(\sum_{j=1}^pe_j\cdot u(e_j)+\sum_{j=1}^p-u(e_j)\cdot e_j\right)$$
with
\begin{eqnarray*}
\sum_{j=1}^p-u(e_j)\cdot e_j&=&-\sum_{i=p+1}^{p+q}\sum_{j=1}^p\langle u(e_j),e_i\rangle\ e_i\cdot e_j\\
&=&\sum_{i=p+1}^{p+q}e_i\cdot\left(-\sum_{j=1}^p\langle e_j,u^*(e_i)\rangle\ e_j\right)\\
&=&\sum_{i=p+1}^{p+q}e_i\cdot (-u^*(e_i)),
\end{eqnarray*}
which gives (\ref{biv u u*}).
\end{proof}
\begin{lem}\label{lem4 ap1}
Let us consider two linear maps $u:\R^p\rightarrow\R^q$ and $v:\R^n\rightarrow\R^n,$ with $v$ skew-symmetric, and the associated bivectors
$$u=\sum_{j=1}^pe_j\cdot u(e_j),\hspace{1cm}v= \frac{1}{2}\sum_{j=1}^ne_j\cdot v(e_j).$$ 
Then $[u,v]\in\Lambda^2\R^n$ represents the map
$$\xi=\xi_p+\xi_q\mapsto -u^*(v(\xi)_q)+v(u^*(\xi_q))+u(v(\xi)_p)-v(u(\xi_p)),$$
where the sub-indices $p$ and $q$ mean that we take the components of the vectors in $\R^p$ and $\R^q$ respectively. In view of Lemma \ref{lem1 ap1}, this may also be written in the form
$$\left[[u,v],\xi\right]=-u^*(v(\xi)_q)+v(u^*(\xi_q))+u(v(\xi)_p)-v(u(\xi_p))$$
for all $\xi\in\R^n.$ 
\end{lem}
\begin{proof}
From Lemmas \ref{lem2 ap1} and \ref{lem3 ap1}, the bivector $[u,v]\in\Lambda^2\R^n$ represents 
$$\left(\begin{array}{cc}0&-u^*\\u&0\end{array}\right)\circ v-v\circ \left(\begin{array}{cc}0&-u^*\\u&0\end{array}\right),$$
that is the map
\begin{eqnarray*}\xi&\mapsto& \left(\begin{array}{cc}0&-u^*\\u&0\end{array}\right)\left(\begin{array}{c}v(\xi)_p\\v(\xi)_q\end{array}\right)-v\left(\begin{array}{cc}0&-u^*\\u&0\end{array}\right)\left(\begin{array}{c}\xi_p\\\xi_q\end{array}\right)\\
&&=\left(\begin{array}{c}-u^*(v(\xi)_q)+v(u^*(\xi_q))\\u(v(\xi)_p)-v(u(\xi_p))\end{array}\right),
\end{eqnarray*}
which gives the result.
\end{proof}
\textbf{Acknowledgments:} P. Bayard was supported by the project PAPIIT-UNAM IA105116. Section \ref{section fundamental theorem} and the study of $Sol_3$ are part of B. Zavala Jimenez's PhD; she thanks CONACYT for support.

\end{document}